\documentclass[a4paper,11pt]{amsart}

\usepackage{graphicx}        
\usepackage{multicol}        
\usepackage[justification=centering]{caption}
\usepackage{microtype}
\usepackage{algorithm}
\usepackage[noend]{algorithmic}
\usepackage{color}

\newtheorem{theorem}{Theorem}[section]
\newtheorem{proposition}[theorem]{Proposition}
\newtheorem{lemma}[theorem]{Lemma}

\newtheorem{definition}[theorem]{Definition}

\newtheorem{remark}[theorem]{Remark}

\newtheorem{example}[theorem]{Example}
\usepackage{amssymb}
\usepackage{amsmath}
\usepackage{amsfonts}
\usepackage{booktabs}
\usepackage{multirow}
\usepackage{exscale}
\usepackage{latexsym}
\usepackage{xspace}
\usepackage{enumerate}
\usepackage{todonotes}
\usepackage{hyperref}
\usepackage{cleveref}
\usepackage{units}
\usepackage{enumitem}
\usepackage{multicol}

\newlist{cvdesc}{description}{1}
\setlist[cvdesc]{nosep,
labelindent=0pt,
labelwidth=2.8cm,
labelsep*=0.2cm,
leftmargin=cm,
font=\normalfont,
align=right}
\newlist{compactenum}{enumerate}{3}
\setlist[compactenum]{nosep,itemsep=2pt,topsep=2pt,labelindent=1em,leftmargin=1em}
\setlist[compactenum,1]{label=\textbullet}

\def\NN{{\mathbb N}}
\def\ZZ{{\mathbb Z}}
\def\QQ{{\mathbb Q}}
\def\KK{{\mathbb K}}

\def\char{{\rm{char}}}

\def\Aut{{\mathrm{End}}}

\def\GL{{\mathrm{GL}}}
\def\S{{\mathrm{S}}}

\def\Gr{Gr\"obner}

\def\lm{{\mathrm{lm}}}
\def\lc{{\mathrm{lc}}}

\def\LM{{\mathrm{LM}}}

\def\HF{{\mathrm{HF}}}
\def\HS{{\mathrm{HS}}}
\def\H{{\mathrm{H}}}

\def\mHS{{\mathrm{mHS}}}
\def\Sol{{\mathrm{Sol}}}
\def\mSol{{\mathrm{mSol}}}
\def\lph{{\mathrm{lpHilbert}}}
\def\mi{{\mathcal{MI}}}

\def\O{{\mathcal{O}}}
\def\N{{\mathcal{N}}}

\def\C{{\mathbf{C}}}
\def\vH{{\mathbf{H}}}
\def\S{{\mathrm{S}}}

\def\bO{\bar{\mathcal{O}}}
\def\bC{\bar{\mathbf{C}}}
\def\bvH{\bar{\mathbf{H}}}
\def\bA{\bar{A}}
\def\bp{\bar{p}}

\def\tdeg{{\mathrm{tdeg}}}
\def\balpha{{\bar{\alpha}}}

\def\bdelta{{\bar{\delta}}}
\def\bDelta{{\bar{\Delta}}}
\def\btheta{{\bar{\theta}}}
\def\bx{{\bar{x}}}
\newcommand{\singular}{{\sc Singular}\xspace}
\newcommand{\magma}{{\sc Magma}\xspace}
\newcommand{\maple}{{\sc Maple}\xspace}

\def\e{{\epsilon}}

\begin{document}

\title[Multigraded Hilbert series $\ldots$]
{Multigraded Hilbert Series of noncommutative modules}

\author[R. La Scala]{Roberto La Scala$^*$}

\author[S.K. Tiwari]{Sharwan K. Tiwari$^{**}$}

\address{$^*$ Dipartimento di Matematica, Universit\`a di Bari,
Via Orabona 4, 70125 Bari, Italy}
\email{roberto.lascala@uniba.it}

\address{$^{**}$ Scientific Analysis Group, Defence Research
\& Development Organization, Metcalfe House, Delhi-110054, India}
\email{shrawant@gmail.com}

\thanks{The author R.L.S. acknowledges the support of the University of Bari.
The author S.K.T. acknowledges the support of the University of Kaiserslautern
and thanks the German Academic Exchange Service (DAAD) for providing the
financial support}

\subjclass[2010] {Primary 16Z05. Secondary 16P90, 05E05}

\keywords{Hilbert series, noncommutative modules, regular languages,
symmetric functions}

\maketitle

\begin{quotation}
{\em Per Maria Jos\'e}
\end{quotation}

\begin{abstract}
In this paper, we propose methods for computing the Hilbert series of
multigraded right modules over the free associative algebra. In particular,
we compute such series for noncommutative multigraded algebras. Using
results from the theory of regular languages, we provide conditions when
the methods are effective and hence the sum of the Hilbert series is a
rational function. Moreover, a characterization of finite-dimensional algebras
is obtained in terms of the nilpotency of a key matrix involved in the computations.
Using this result, efficient variants of the methods are also developed for
the computation of Hilbert series of truncated infinite-dimensional algebras
whose (non-truncated) Hilbert series may not be rational functions.
We consider some applications of the computation of multigraded Hilbert series
to algebras that are invariant under the action of the general linear group.
In fact, in this case such series are symmetric functions which can be decomposed
in terms of Schur functions. Finally, we present an efficient and complete
implementation of (standard) graded and multigraded Hilbert series that has
been developed in the kernel of the computer algebra system \singular.
A large set of tests provides a comprehensive experimentation for the proposed
algorithms and their implementations.
\end{abstract}

\section{Introduction}

The concept of what is nowadays known as Hilbert (or Hilbert-Poincar\'e) series,
was first introduced in the 19th century in the context of finitely generated
commutative algebras over a field $\KK$. This is the generating series
of the Hilbert function which relates each nonnegative integer $d$ to the $\KK$-linear
dimension of the graded or filtered component of degree $d$ of the algebra.

A fundamental property of a finitely generated commutative algebra is that the sum
of its Hilbert series is always a rational function allowing a finite description
of such an invariant. This was proved by Hilbert himself in 1890 by introducing
the other fundamental notion of a free resolution. Indeed, to compute
a resolution, that is, a complete chain of syzygies for the generators
of the ideal of relations of the algebra is usually much more involved than
to determine the Hilbert series. This was essentially observed in 1927 by Francis
Sowerby Macaulay, whose results imply that the Hilbert series of a commutative algebra
is equal to the series of a corresponding monomial algebra, that is, an algebra
whose generators are related by monomials. In modern terminology, such a monomial
algebra is defined by the leading monomial ideal, with respect to a suitable monomial
ordering, of the ideal of relations of the algebra under consideration. In other words,
one has to compute a \Gr\ basis of such an ideal. In fact, it was only around the
mid-1990s when the computer algebra community was able to propose and implement
many efficient algorithms for the computation of the Hilbert series of commutative
algebras (see, for instance, \cite{BS,Bi}).

For noncommutative structures, the first result about the rationality of the Hilbert
series of a finitely presented monomial algebra is due to Govorov \cite{Go} and
dates back only to 1972. Another fundamental contribution was given by Ufnarovski
\cite{Uf1} who developed a graph theoretic method to study the asymptotic behaviour
of the Hilbert functions of finitely presented monomial algebras and also to compute
the corresponding rational Hilbert series. We mention that in the noncommutative
literature, Hilbert functions and Hilbert series are also called ``growth functions
and growth series''. Clearly, these results of Govorov and Ufnarovski can be immediately
extended to (finitely generated) algebras whose ideals of relations admit a finite \Gr\ basis.
However, the free associative algebra $\KK\langle x_1,\ldots,x_n\rangle$ is not
Noetherian which implies that noncommutative \Gr\ bases are generally infinite sets.

Recently, in \cite{LS2} it has been proved that the rationality of a noncommutative
Hilbert series follows from the condition that the (possibly infinitely generated)
leading monomial ideal of relations of the algebra is a ``regular language'' in the sense
of the theory of formal languages. In fact, this result has been stated in the more
general context of finitely generated right modules over
$\KK\langle x_1,\ldots,x_n\rangle$. Moreover, in that paper one finds an iterative
method for the computation of the Hilbert series which is proved to be optimal with
respect to the number of iterations by means of standard results in automata theory.
Since this method is based on the repeated application of the (right) colon ideal
operation, from a historical perspective we can say that it generalizes and refines
the methods that were developed for commutative Hilbert series, indicating that
the worlds of commutative and noncommutative algebra are not far apart.

In \cite{LS2} the only implementation of the proposed algorithms was an
experimental one in \maple. Then, our first goal in the present paper is to develop
a fast and well-tested implementation of noncommutative Hilbert series
in the kernel of \singular \cite{DGPS} (see Section \ref{timings} for details).
We remark that this is the only general implementation
of these series up to now. In fact, the library \textsf{fpadim.lib} of \singular
provides the computation of the Hilbert function of a finite-dimensional algebra
by means of the enumeration of normal words. In our tests we propose hence
the comparison of the two methods for the finite-dimensional case.

Our second goal is to extend the theory and the methods to
multivariate Hilbert series that are defined for multigraded algebras.
It is well known that if an algebra with $n$ generators over a field $\KK$
of characteristic zero is also a (polynomial) module for the action of the
general linear group $\GL_n(\KK)$, then it is a multigraded algebra.
An algebra is a $\GL_n(\KK)$-module when the ideal of relations of its
$n$ generators is invariant under all invertible linear substitutions of them.
This is the case for many universal enveloping algebras and also for the algebras
that are defined by ``T-ideals'', that is, ideals that are invariant under
all polynomial substitutions of the generators.
For $\GL_n(\KK)$-invariant algebras, the multigraded Hilbert series are in fact
symmetric functions. The decomposition of these functions in terms of Schur functions
provides all essential information (multiplicities) about the decomposition
of the algebra in terms of its simple $\GL_n(\KK)$-submodules.
For the purposes of representation theory, this is of course a very important task.
The general theory for the multigraded Hilbert series of finitely generated
multigraded right modules over $\KK\langle x_1,\ldots,x_n\rangle$ is presented
in Section \ref{general} and \ref{macaulay}. In particular, we prove
Theorem \ref{macth} which is a noncommutative multigraded version
of Macaulay's basis theorem that reduces the computations to the monomial case.
In Section \ref{method} and \ref{regular} we present the methods
by extending the approach that has been introduced in \cite{LS2} for the graded case.
In Section \ref{example} we apply the computation of the multigraded Hilbert
series to an interesting $\GL_n(\KK)$-invariant algebra with the purpose of obtaining
the Schur function decomposition of its series. For the reader's convenience,
a brief review of the very basic theory of polynomial representations
of the general linear group is provided at the beginning of this section.

A last contribution of the present paper consists in developing efficient
variants of the proposed methods that can be applied to truncations of
infinite-dimensional algebras up to some fixed degree. Note that the truncations
yield finite-dimensional algebras and hence their Hilbert series are in fact
polynomials. The motivation for developing such variants is two-fold.
First, there exist non-regular monomial algebras for which the sum of the
corresponding Hilbert series is not available. Thus, polynomial approximations
of these functions may be useful to understand them. Moreover,
since Schur functions are (symmetric) polynomials, there are fast algorithms
to perform the Schur function decomposition for symmetric polynomials but it is
much more difficult to decompose even a rational symmetric function as an infinite
sum of Schur functions. Again, to have approximations of the latter decomposition
may help to understand the complete picture. The truncation methods are
presented in Section \ref{truncation} and applied in Section \ref{example}
and Section \ref{timings}. In Section \ref{truncation} one also finds
a characterization of the finite-dimensionality of a (regular) monomial
algebra in terms of the nilpotency of a key matrix that is involved
in our method. Finally, in Section \ref{conclusion} we draw some conclusions
and suggest further developments of the theory and the methods.

\section{Multigraded modules and their Hilbert series}
\label{general}

Let $\KK$ be any field and let $X = \{x_1,\ldots,x_n\}$ be a finite set.
We denote by $W = X^*$ the free monoid that is freely generated
by $X$, that is, the elements of $W$ are words over the alphabet $X$.
Moreover, we denote by $F = \KK\langle X\rangle$ the corresponding monoid
$\KK$-algebra, that is, $F$ is the (finitely generated) free associative
algebra freely generated by $X$. A {\em monomial} of $F$ is by definition
a word of $W$ and an element of $F$ is called a {\em (noncommutative) polynomial}.
Recall that a standard grading of the algebra $F$ is given by assigning
$\tdeg(w)$ as the length of a word $w\in W$. We call also $\tdeg(w)$
the {\em total degree} of the monomial $w$. Thus, we have that
$F = \bigoplus_{d\in\NN} F_d$ where $F_d$ is the span of the words
$w\in W$ such that $\tdeg(w) = d$. An element $f\in F_d$ is called
a {\em homogeneous polynomial of total degree $d$}.
A standard multigrading of $F$ is defined in the following way.

\begin{definition}
For any $w\in W$, we define $\deg(w) = \balpha = (\alpha_1,\ldots,\alpha_n)$,
where $\alpha_i\in\NN$ is the number of times the variable $x_i$ occurs
in the word $w$. Since $\deg(v w) = \deg(v) + \deg(w)\in\NN^n$ for all $v,w\in W$,
we have an algebra multigrading
\[
F = \bigoplus_{\balpha\in\NN^n} F_{\balpha}
\]
where $F_{\balpha}$ is the subspace of $F$ that is spanned by the words $w\in W$
such that $\deg(w) = \balpha$. Then, an element $f\in F_{\balpha}$ is called a
{\em multihomogeneous polynomial of multidegree $\balpha$}.
\end{definition}

A variant of the above multigrading can be obtained in the following way.
Consider $F$ as the cyclic free right $F$-module generated by $1$.
Fix a multidegree $\bdelta = (\delta_1,\ldots,\delta_n)\in \NN^n$ and define
$\deg'(1) = \bdelta$, that is, $\deg'(w) = \bdelta + \deg(w)$ for each $w\in W$.
Then, we denote by $F[-\bdelta]$ the cyclic free right $F$-module $F$ that is
endowed by the multigrading defined by $\deg'$. By denoting $\balpha \unrhd \bdelta$
when $\alpha_i\geq \delta_i$ for any $i = 1,2,\ldots,n$, we have that
\[
F[-\bdelta]_\balpha =
\left\{
\begin{array}{cl}
F_{\balpha-\bdelta} & \mbox{if}\ \balpha\unrhd \bdelta, \\
0 & \mbox{otherwise}.
\end{array}
\right.
\]
Let now $r > 0$ be any integer and let $F^r$ be the free right $F$-module of rank $r$.
If $\{e_1,\ldots,e_r\}$ denotes the canonical basis of $F^r$, then the elements
of $F^r$ are right linear combinations $\sum_i e_i f_i$ where $f_i\in F$.
A standard right module multigrading of $F^r$ is defined by putting $\deg(e_i) =
\deg(1) = (0,\ldots,0)$, that is, $\deg(e_i w) = \deg(w)$ for all $w\in W$ and
$i = 1,2,\ldots,r$. By using a set of multidegrees $\{\bdelta_1,\ldots,\bdelta_r\}$
we can modify this multigrading of $F^r$ by putting $\deg'(e_i) = \bdelta_i$,
for any $i$. We will denote this multigraded (free) right $F$-module as
$\bigoplus_{1\leq i\leq r} F[-\bdelta_i]$. In fact, an element
$f = \sum_i e_i f_i\in \bigoplus_i F[-\bdelta_i]$ is multihomogeneous
of multidegree $\balpha$ if and only if $f_i\in F[-\bdelta_i]_\balpha$,
for all $i = 1,2,\ldots,r$.

\begin{definition}
Consider a right submodule $M\subset \bigoplus_i F[-\bdelta_i]$. We call $M$ a
{\em multigraded submodule} if $M = \sum_\balpha M_\balpha$ where $M_\balpha =
M\cap \bigoplus_i F[-\bdelta_i]_\balpha$. In this case, one has the quotient
multigraded right module $N = \bigoplus_i F[-\bdelta_i]/M$ where the multihomogeneous
component $N_\balpha$ is isomorphic to $\bigoplus_i F[-\bdelta_i]_\balpha/M_\balpha$.
A finitely generated multigraded right module $N' = \langle g_1,\ldots,g_r \rangle$
where $g_i$ is a multihomogeneous element of multidegree $\bdelta_i$ is clearly
isomorphic to $N$ by the multigraded right module homomorphism
$\varphi:\bigoplus_i F[-\bdelta_i]\to N', e_i\mapsto g_i$ where $M = \ker\varphi$.
\end{definition}

Let $\balpha = (\alpha_1,\ldots,\alpha_n)\in\NN^n$ and put $d = |\balpha| =
\sum_i \alpha_i$. By counting the number of words that have multidegree $\balpha$,
one has that
\[
\dim F_\balpha = \binom{d}{\alpha_1,\ldots,\alpha_n} = \frac{d!}{\alpha_1!\cdots\alpha_n!}
\]
Recall that $(t_1 + \cdots + t_n)^d = \sum_{|\balpha| = d} \binom{d}{\alpha_1,\ldots,\alpha_n}
t_1^{\alpha_1}\cdots t_n^{\alpha_n}$ by the multinomial theorem, which implies that
the multivariate generating series of the function $\balpha\mapsto \dim F_\balpha$
satisfies the following formula
\[
\sum_{\balpha\in\NN^n} \dim F_\balpha\, t_1^{\alpha_1}\cdots t_n^{\alpha_n} =
\sum_{d\in\NN} (t_1 + \cdots + t_n)^d = \frac{1}{1 - (t_1 + \cdots + t_n)}.
\]
Moreover, for any fixed multidegree $\bdelta = (\delta_1,\ldots,\delta_n)$, we have
by definition $F[-\bdelta]_\balpha = F_{\balpha-\bdelta}$ and therefore
\[
\sum_\balpha \dim F[-\bdelta]_\balpha\, t_1^{\alpha_1}\cdots t_n^{\alpha_n} =
\frac{t_1^{\delta_1}\cdots t_n^{\delta_n}}{1 - (t_1 + \cdots + t_n)}.
\]
Next, we will generalize the notions and results above.

\begin{definition}
Let $N = \bigoplus_{\balpha} N_{\balpha}$ be a finitely generated multigraded
right module over $F$. We define the function $\HF(N)(\balpha) = \dim N_\balpha$,
for any $\balpha = (\alpha_1,\ldots,\alpha_n)\in\NN^n$ and the corresponding
multivariate generating series
\[
\HS(N) = \sum_{\balpha\in\NN^n} \dim N_\balpha \, t_1^{\alpha_1}\cdots t_n^{\alpha_n}
\]
We call $\HF,\HS$ respectively the {\em multigraded Hilbert function and multigraded
Hilbert series of the module $N$}.
\end{definition}

Observe that the {\em (graded) Hilbert function of $N$} is defined for any total
degree $d\in\NN$ as
\[
\HF'(N)(d) = \sum_{|\balpha| = d} \HF(N)(\balpha)
\]
In other words, the corresponding {\em (graded) Hilbert series} $\HS'(N)$ is obtained
by identifying in the series $\HS(N)$ all variables $t_i$ ($1\leq i\leq n$) with
a single variable $t$.

We have already obtained explicit formulas for multigraded Hilbert function and
series in the case of free algebras. Of course, these formulas immediately extend
to finitely generated free right modules. Note now that an important property
of the finitely generated free associative algebra $F = \KK\langle X \rangle$
is that it is a {\em free right ideal ring} \cite{Co}, that is, each right
submodule $M\subset F^r$ is in fact a free one of unique rank. Unfortunately,
owing to non-Noetherianity of the algebra $F$, we cannot always assume that $M$
is also finitely generated. Nevertheless, for this case one has the following
formula.

\begin{theorem}
\label{fgformula}
Let $N = \bigoplus_{1\leq i\leq r} F[-\bdelta_i]/M$ be a finitely presented multigraded
right module where $\bdelta_i = (\delta_{i1},\ldots,\delta_{in})$, for all $i$.
Consider $\{g_1,\ldots,g_s\}$ a multihomogeneous free basis of the right submodule
$M\subset \bigoplus_i F[-\bdelta_i]$. If $\bDelta_j = (\Delta_{j1},\ldots,\Delta_{jn})
= \deg(g_j)$ then
\[
\HS(N) = \frac{\sum_{1\leq i\leq r} t_1^{\delta_{i1}}\cdots t_n^{\delta_ {in}} -
\sum_{1\leq j\leq s} t_1^{\Delta_{j1}}\cdots t_n^{\Delta_ {jn}}}{1 - (t_1 + \cdots + t_n)}
\]
In particular, we have that $\HS(N)$ is a rational function with integer
coefficients.
\end{theorem}

\begin{proof}
Consider $\{\e_1,\ldots,\e_s\}$ the canonical basis of the free right module
$\bigoplus_{1\leq j\leq s} F[-\bDelta_j]$ and define the multigraded right module
homomorphism
\[
\bigoplus_j F[-\bDelta_j]\to \bigoplus_i F[-\bdelta_i], \e_j\mapsto g_j.
\]
Since $\{g_j\}$ is a free basis of $M$, the map above implies a short exact
sequence
\[
0 \to \bigoplus_j F[-\bDelta_j]\to \bigoplus_i F[-\bdelta_i]\to N\to 0.
\]
Because all the homomorphisms are multigraded ones, we obtain that
\[
\HS(\bigoplus_j F[-\bDelta_j]) - \HS(\bigoplus_i F[-\bdelta_i]) + \HS(N) = 0
\]
which implies the stated formula.
\end{proof}

It is clear that for finitely generated but infinitely presented right modules
the above result does not apply and hence we have to follow a different path.
In the next section we begin by reducing the problem of determining multigraded
Hilbert series to the case of monomial cyclic right modules.

\section{Monomial cyclic right modules}
\label{macaulay}

Let $F^r$ be a free right module and consider $\{e_1,\ldots,e_r\}$ its canonical basis.
We denote $W(r) = \cup_{i=1}^r e_i W = \{e_i w\mid 1\leq i\leq r, w\in W\}$ which is
a canonical $\KK$-linear basis of $F^r$. The elements of $W(r)$ are called the
{\em monomials of $F^r$}. 

\begin{definition}
Let $\prec$ be a well-ordering of $W(r)$. We call $\prec$ a {\em monomial ordering}
of $F^r$ if $\prec$ is compatible with the right module structure of $F^r$, that is,
for all $e_i u, e_j v\in W(r)$ and $w\in W$, one has that
\[
e_i u\prec e_j v \Rightarrow e_i u w\prec e_j v w.
\]
\end{definition}

There are well-known examples of monomial orderings of $F^r$. For instance, for any
$e_i v, e_j w\in W(r)$ we can define that $e_i v\prec e_j w$ if and only if either
$v < w$ in some graded lexicographic ordering or $v = w$ and $i < j$. Assume now that
$F^r$ is endowed with a monomial ordering $\prec$. For any element $f\in F^r$
we denote by $\lm(f)$ the greatest, with respect to $\prec$, among the monomials
occurring in $f$. The element $\lm(f)\in W(r)$ is called the {\em leading monomial
of $f$}. We denote in addition by $\lc(f)\in\KK$ the coefficient that the monomial
$\lm(f)$ has in $f$ and we call it the {\em leading coefficient of $f$}.
If $M\subset F^r$ is a right submodule, then we denote by $\LM(M)$ the right
submodule of $F^r$ that is generated by the set
$\lm(M) = \{\lm(f)\mid f\in M,f\neq 0\}\subset W(r)$. Note that $\LM(M)$
is a {\em monomial right submodule} which means that it is generated by monomials
of $F^r$. Then, we call $\LM(M)$ the {\em leading monomial module of $M$}.
It is well known that to compute a minimal set of monomials generating $\LM(M)$
one uses the notion of (minimal) \Gr\ basis \cite{LSL1,LS1}. Although these bases
are generally infinite because of the non-Noetherianity of the free associative
algebra $F$, in many cases they can be described in closed form by the
aid of partial computations and formal arguments (see, for instance, \cite{DLS}).

We now state a generalization of Macaulay's basis theorem for commutative modules
in the noncommutative setting.

\begin{theorem}
\label{macth}
Let $N = \bigoplus_{1\leq i\leq r} F[-\bdelta_i]/M$ be a finitely generated
multigraded right module. Fix any monomial ordering for the free right module $F^r$
and consider $N' = \bigoplus_{1\leq i\leq r} F[-\bdelta_i]/\LM(M)$. Then, we have
that $\HF(N) = \HF(N')$ and hence $\HS(N) = \HS(N')$.
\end{theorem}

\begin{proof}
Put $M' = \LM(M)$ and denote $W(r)_\balpha =
W(r)\cap \bigoplus_i F[-\bdelta_i]_\balpha$, for any multidegree $\balpha\in\NN^n$.
Since $M'\subset \bigoplus_i F[-\bdelta_i]$ is a monomial module, a linear basis
of $\bigoplus_i F[-\bdelta_i]_\balpha/M'_\balpha$ is clearly given by
the set $\{e_i w + M'_\balpha\mid e_i w\in W(r)_\balpha\setminus M'\}$.
We have to prove that the corresponding set $\{e_i w + M_\balpha\mid
e_i w\in W(r)_\balpha\setminus M'\}$ is a linear basis of
$\bigoplus_i F[-\bdelta_i]_\balpha/M_\balpha$. Consider any multihomogeneous
element $f\in \bigoplus_i F[-\bdelta_i]_\balpha, f\neq 0$. If $\lm(f)\in M'$
then there exists a multihomogeneous element $g\in M$ such that $\lm(f) = \lm(g) v$,
for some $v\in W$. By putting $f_1 = f - g v \frac{\lc(f)}{\lc(g)}$ we obtain that
either $f \equiv f_1$ mod $M_\balpha$ with $f_1 = 0$ or $\lm(f)\succ \lm(f_1)$. 
In the latter case, we can repeat this division step for the multihomogeneous element
$f_1\in \bigoplus_i F[-\bdelta_i]_\balpha, f_1\neq 0$. Since $W(r)_\balpha$
is a finite set, we conclude that either $f \equiv f_2$ mod $M_\balpha$, for some
multihomogeneous element $f_2\in \bigoplus_i F[-\bdelta_i]_\balpha$ such that
$f_2 = 0$ or $\lm(f)\succ \lm(f_2)\notin M'$. If $f_2\neq 0$, then we consider
the multihomogeneous element $f_3 = f_2 - \lm(f_2)\lc(f_2)$ so that
$f - \lm(f_2)\lc(f_2)\equiv f_3$ mod $M_\balpha$. Note that one has either $f_3 = 0$
or $\lm(f_2)\succ \lm(f_3)$. By iterating the division process, we finally obtain
that $f\equiv f'$ mod $M_\balpha$ where $f' = \sum_k e_{i_k} w_k c_k$ with
$e_{i_k} w_k\in W(r)_\balpha\setminus M'$ and $c_k\in \KK$. Moreover, it is clear
that $f'\in M_\balpha$ if and only if $f' = 0$.
\end{proof}

The result above implies that the problem of computing a multigraded Hilbert series
can be reduced to the case of a finitely generated multigraded right module
$N = \bigoplus_i F[-\bdelta_i]/M$ where $M\subset\bigoplus_i F[-\bdelta_i]$ is a
monomial right submodule. In this case, we call $N$ a (finitely generated)
{\em monomial right module}. For these modules we are immediately reduced to
the cyclic case by means of the following result.

\begin{proposition}
\label{cycdec}
Let $N = \bigoplus_{1\leq i\leq r} F[-\bdelta_i]/M$ be a finitely generated
monomial right module and denote by $\{e_{i_k} w_k\}\subset W(r)$ a monomial
generating set of $M$. For each index $i = 1,2,\ldots,r$, let
$I_i\subset F[-\bdelta_i]$ be the monomial right ideal that is generated by
set $\{w_k\mid e_{i_k} = e_i\}$. Moreover, let $C_i = F[-\bdelta_i]/I_i$ be
the corresponding monomial cyclic right module. Thus, one has that
$M = \bigoplus_{i=1}^r e_i I_i$ and hence $N$ is isomorphic to
$\bigoplus_{i=1}^r C_i$. For the multigraded Hilbert series this implies that
\[
\HS(N) = \sum_{i=1}^r \HS(C_i).
\]
\end{proposition}

We recall that $F$ is a free right ideal ring which implies that a minimal
basis of a monomial right ideal $I$ is in fact a free one. If this basis is finite,
then one can immediately obtain the multigraded Hilbert series of the cyclic module
$C = F[-\bdelta]/I$ by means of the formula in Theorem \ref{fgformula}.
Unfortunately, this happens very seldom. For instance, even if $I$ is a finitely
generated two-sided ideal, it may be infinitely generated as a right ideal.
In the next section, to solve the problem of computing $\HS(C)$ in the general
case, we develop an iterative method which relates this series to the Hilbert series
of monomial cyclic right modules that are obtained from $C$.

\section{A key linear equation}
\label{method}

Let $C = F[-\bdelta]/I$ be a monomial cyclic right module. For any
$i = 1,2,\ldots,n$, denote by $\btheta_i$ the multidegree that the variable
$x_i$ has in $F$, that is, $\btheta_i = (0,\ldots,0,1,0,\ldots,0)$ where
1 occurs in the $i$-th position. Put $\bdelta_i = \bdelta + \btheta_i$ and observe
that $x_i$ has multidegree $\bdelta_i$ in $F[-\bdelta]$. Let $\{e_1,\ldots,e_n\}$
be the canonical basis of the multigraded free right module
$\bigoplus_{1\leq i\leq n} F[-\bdelta_i]$ and denote $\bx_i = x_i + I\in C$.
We define the multigraded right module homomorphism
\[
\varphi:\bigoplus_i F[-\bdelta_i]\to C, e_i\mapsto \bx_i.
\]
By definition, the image of this map is the multigraded right submodule 
$B = \langle \bx_1,\ldots,\bx_n \rangle\subset C$. Thus, the cokernel $C/B$ 
is either zero if $C = 0$ or it is isomorphic to the base field $\KK$ otherwise.
To compute the kernel of $\varphi$ one has to consider the colon right ideal
\[
(I :_R x_i) = \{f\in F\mid x_i f\in I\}
\]
which is also a monomial right ideal of $F$ (see \cite{LS2} for more details).
By putting $I_{x_i} = (I :_R x_i)$, we consider the monomial cyclic right
module $C_{x_i} = F[-\bdelta]/I_{x_i}$ and we denote $C_{x_i}[-\btheta_i] =
F[-\bdelta_i]/I_{x_i}$. Since $I$ is a monomial ideal, one has immediately that
\[
\ker\varphi = \bigoplus_i e_i I_{x_i}\subset \bigoplus_i F[-\bdelta_i].
\]
Therefore, we obtain the following short exact sequence of multigraded right
module homomorphisms
\[
0\to \bigoplus_{i=1}^n C_{x_i}[-\btheta_i]\to C\to C/B\to 0.
\]
For the corresponding multigraded Hilbert series, since one has clearly that
$\HS(C_{x_i}[-\btheta_i]) = t_i\cdot \HS(C_{x_i})$, we obtain the following key
linear equation
\begin{equation}
\label{keyeq}
\HS(C) = \sum_{i=1}^n t_i\cdot \HS(C_{x_i}) + c(I)
\end{equation}
where, by definition, $c(I)$ is the dimension of the cokernel $C/B$, that is
\[
c(I) =
\left\{
\begin{array}{cl}
0 & \mbox{if}\ I = \langle 1 \rangle, \\
1 & \mbox{otherwise.}
\end{array}
\right.
\]
In view of the equation (\ref{keyeq}), one may think to reduce the computation of the
multigraded Hilbert series $\HS(C)$ to the one of the series $\HS(C_{x_i})$
($1\leq i\leq n$) and iteratively to the computation of $\HS(C_{x_ix_j})$
($1\leq i,j\leq n$) and so on. It may happen that this process terminates
in a finite number of steps because many of these monomial cyclic right modules
coincide with each other. In the next section, we explain why and how,
in this case, the series $\HS(C)$ can be immediately obtained.

\section{Orbits of monomial right ideals}
\label{regular}

In this section, it is essential to recall the following two definitions
that have been used in \cite{LS2}.

\begin{definition}
Let $\mi$ denote the set of all monomial right ideals of $F$. For all variables
$x_i$ $(1\leq i\leq n)$ we let $T_{x_i}$ be the colon right ideal operator
on $\mi$ defined by $x_i$, namely $T_{x_i}(I) = (I :_R x_i)$ for any $I\in\mi$.
Moreover, we denote by $\O_I$ the minimal subset of $\mi$ containing $I$
such that $T_{x_i}(\O_I)\subset \O_I$ for any variable $x_i$.
The set $\O_I$ is called the {\em orbit of $I\in\mi$}. A monomial right ideal $I$
is called {\em regular} if its orbit $\O_I$ is a finite set.
\end{definition}

\begin{definition}
Consider $I\subset F$ a regular (monomial) right ideal and let its orbit
$\O_I = \{I_1,\ldots,I_r\}$ be an ordered set where $I_1 = I$. Define
a square matrix $A_I = (a_{kl})\in\ZZ^{r\times r}$ such that
\[
a_{kl} =\# \{1\leq i\leq n\mid T_{x_i}(I_k) = I_l\}.
\]
Let $E_r\in\ZZ^{r\times r}$ be the identity matrix and consider the field
of rational functions $\QQ(t)$ in the variable $t$ and with coefficients in $\QQ$
(in fact in $\ZZ$). We denote $p_I(t) = \det(t\cdot E_r - A_I) \in\QQ(t)$
the characteristic polynomial of $A_I$. Finally, define the column vector
$\C_I = (c(I_1),\ldots,c(I_r))^t$. We call $A_I, p_I(t)$ and $\C_I$ respectively,
the {\em adjacency matrix, characteristic polynomial and constant vector
of (the orbit of) $I$}.
\end{definition}

Observe that $\det(E_r - t\cdot A_I) = t^r p_I(1/t)\neq 0$,
since $E_r - t\cdot A_I = t (\frac{1}{t}\cdot E_r - A_I)\in \QQ(t)^{r\times r}$.
We now introduce a new set of matrices.

\begin{definition}
Let $I$ be a regular right ideal. For each index $i = 1,2,\ldots,n$, we denote by
$A_I^{(i)}\in \ZZ^{r\times r}$ the square matrix that is defined as follows
\[
A_I^{(i)} = (a_{kl}^{(i)}) \,,\, a_{kl}^{(i)} =
\left\{
\begin{array}{cl}
1 & \mbox{if}\ T_{x_i}(I_k) = I_l, \\
0 & \mbox{otherwise.}
\end{array}
\right.
\]
By definition, one has that $A_I = A_I^{(1)} + \cdots + A_I^{(n)}$ and therefore
we call $A_I^{(i)}$ the {\em $i$-th component of the adjacency matrix $A_I$}.
\end{definition}

Let $\QQ(t_1,\ldots,t_n)$ be the field of rational functions in the variables
$t_1,\ldots,t_n$ and with coefficients in $\QQ$. We may consider the square
matrix $E_r - \sum_i t_i\cdot A_I^{(i)}$ as an element of the matrix algebra
$\QQ(t_1,\ldots,t_n)^{r\times r}$.

\begin{lemma}
\label{invmat}
$E_r - \sum_i t_i\cdot A_I^{(i)}$ is an invertible matrix.
\end{lemma}

\begin{proof}
Consider the algebra homomorphism $\QQ[t_1,\ldots,t_n]\to\QQ[t]$
such that $t_i\mapsto t$, for any $i=1,2,\ldots,n$. Clearly,
this homomorphism maps $\det(E_r - \sum_i t_i\cdot A_I^{(i)})$ into
$\det(E_r - t\cdot A_I)$. Since we have already observed that the latter
determinant is different from zero, we conclude that the same holds
for the former one.
\end{proof}

We now show that, in the regular case, the multigraded Hilbert series of a monomial
cyclic right module can be obtained by solving a linear system corresponding to the
non-singular matrix $E_r - \sum_i t_i\cdot A_I^{(i)}$.

\begin{theorem}
\label{mainth}
Let $C = F[-\delta]/I$ be a {\em regular monomial cyclic right module}, that is,
$I$ is a regular ideal. Then, the multigraded Hilbert series $\HS(C)$ belongs to
the rational function field $\QQ(t_1,\ldots,t_n)$.
\end{theorem}

\begin{proof}
Let $\O_I = \{I_1,\ldots,I_r\}$ be the finite orbit of $I$ and consider
the monomial cyclic right module $C_k = F[-\delta]/I_k$, for all
$k = 1,2,\ldots,r$. Note that all $I_k$ are in fact regular ideals, because
by definition $\O_{I_k}\subset \O_I$. For each module $C_k$, the linear
equation (\ref{keyeq}) becomes
\begin{equation}
\label{keysys}
\HS(C_k) = \sum_{i=1}^n t_i \cdot \HS(C_{l_{ki}}) + c(I_k)
\end{equation}
where the index $l_{ki}\in\{1,2,\ldots,r\}$ is defined as $T_{x_i}(I_k) =
I_{l_{ki}}$. The matrix form of this linear system with coefficients
in $\QQ(t_1,\ldots,t_n)$ is clearly
\[
(E_r - \sum _{i=1}^n t_i \cdot A_I^{(i)})\vH = \C_I
\]
where $\vH = (\HS(C_1),\ldots,\HS(C_r))^t$ is a column vector of unknown
multigraded Hilbert series. From Lemma \ref{invmat} it follows that the linear
system (\ref{keysys}) has a unique solution, that is, all series $\HS(C_k)$
($C = C_1$) are rational functions with integer coefficients.
\end{proof}

On account of Theorem \ref{macth} and Proposition \ref{cycdec}, the result above
can be immediately generalized in the following way.

\begin{theorem}
Let $N = \bigoplus_{1\leq i\leq r} F[-\bdelta_i]/M$ be a finitely generated
multigraded right module. Fix any monomial ordering for the free right module $F^r$
and consider $N' = \bigoplus_{1\leq i\leq r} F[-\bdelta_i]/M'$ where $M' = \LM(M)$
is the leading monomial module of $M$. According to Proposition \ref{cycdec},
the finitely generated monomial right module $N'$ is isomorphic to a direct sum
$\bigoplus_i C_i$ where $C_i = F[-\bdelta_i]/I_i$ and $I_i$ $(1\leq i\leq r)$
is a monomial right ideal such that $M'$ is isomorphic to $\bigoplus_i I_i$.
If each $I_i$ is a regular ideal, then the multigraded Hilbert series $\HS(N) =
\sum_{1\leq i\leq r} \HS(C_i)$ is a rational function with integer coefficients.
\end{theorem}

Observe now that two problems need to be solved in order to use
the arguments of Theorem \ref{mainth} as an effective method
to compute multigraded Hilbert series. One question is how to compute
the colon right ideal operators $T_{x_i}$ and the second consists in providing
an ``internal characterization'' of the property that a monomial right ideal $I$
is regular, that is, the orbit $\O_I$ is a finite set. Both these problems
have been solved in \cite{LS2} and, for the sake of completeness,
we recall briefly here the corresponding results. Recall that we denote by $W$
the set of monomials of the free associative algebra $F = \KK\langle X \rangle$,
that is, $W$ is the set of all words over the alphabet $X = \{x_1,\ldots,x_n\}$.
Note that for computing the orbit $\O_I$ one needs to iteratively apply
the operators $T_{x_i}$ on the monomial right ideal $I$ until one obtains
a stable set. In other words, for any monomial $w = x_{i_1}\cdots x_{i_d}\in W$,
we define
\[
T_w(I) =  (T_{x_{i_d}}\cdots T_{x_{i_1}})(I) =
(((I :_R x_{i_1}) \cdots ) :_R x_{i_d}) = (I :_R w),
\]
where $(I :_R w) = \{f\in F\mid w f\in I\}$. Note that $w\in I$ if and only if
$(I :_R w) = \langle 1 \rangle$.

\begin{proposition}[\cite{LS2}]
\label{colop}
Consider a monomial right basis $\{w_j\}\subset W$ of a monomial right ideal
$I\subset F$ and let $w\in W$. For all $j$, we denote
\begin{equation*}
w'_j =
\left\{
\begin{array}{cl}
 1  & \mbox{if}\ w = w_j v_j\ (v_j\in W), \\
v_j & \mbox{if}\ w_j = w v_j, \\
 0  & \mbox{otherwise}.
\end{array}
\right.
\end{equation*}
Then, the subset $\{w'_j\}\subset W$ is a monomial right basis of $T_w(I) = (I :_R w)$.
\end{proposition}

An important case is when $I$ is a monomial two-sided ideal, that is, $A = F/I$
is a monomial algebra. It is easy to prove that a two-sided ideal is always contained
in the colon right ideals it defines.

\begin{proposition}[\cite{LS2}]
\label{colbas}
Consider a monomial two-sided basis $\{w_j\}\subset W$ of a monomial two-sided ideal
$I\subset F$ and let $w\in W, w\notin I$. For all $j$, we define the monomial right ideal
\[
R(w,w_j) = \langle v_{jk}\mid u_{jk} w_j = w v_{jk}, u_{jk},v_{jk}\in W,
\tdeg(v_{jk}) < \tdeg(w_j) \rangle.
\]
Then, one has that $(I :_R w) = \sum_j R(w,w_j) + I$.
\end{proposition}

We now recall some notions from the theory of formal languages (see, for instance,
\cite{DLV}).

\begin{definition}
\label{reglangdef}
Any subset $L\subset W$ is called a {\em (formal) language}. Given two languages
$L,L'\subset W$, we consider their set-theoretic union $L\cup L'$ and their
product $L\cdot L' = \{w w'\mid w\in L,w'\in L'\}$. Moreover, one defines
the {\em star operation} $L^* = \bigcup_{d\geq 0} L^d$, where $L^0 = \{1\}$
and $L^d = L\cdot L^{d-1}$ for $d\geq 1$. The union, the product and the star
operation are called the {\em rational operations} over the languages.
A language $L\subset W$ is called {\em regular} if it can be obtained from
finite languages by applying a finite number of rational operations.
\end{definition}

A characterization of the finiteness of the orbit of a monomial right ideal is
provided by the following key result.

\begin{theorem}[\cite{LS2}]
\label{reglang}
The monomial right ideal $I\subset F$ is regular if and only if the corresponding
language $L = I\cap W$ is regular.
\end{theorem}

By using standard results of the theory of finite-state automata, which is strictly
related to the theory of regular languages, in \cite{LS2} it was proved that
the number of variables in the key linear system (\ref{keysys}) is minimal to obtain
the Hilbert series of the monomial cyclic right module under consideration.
In other words, the proposed method for the computation of such a series (both graded
and multigraded) is optimal with respect to the required number of iterations
of the colon right ideal operators. This is essentially a promise of efficiency
that we are able to verify in practice in Section \ref{timings}.

If a right or two-sided monomial ideal is finitely generated, then this ideal is clearly
regular by means of Theorem \ref{reglang}. For instance, if $\{w_1,\ldots,w_k\}\subset W$
is a finite basis of the two-sided monomial ideal $I$, then
\[
I\cap W = W\cdot \{w_1,\ldots,w_k\}\cdot W
\]
where clearly $W = X^*$. In other words, the language $I\cap W$ is obtained
by applying a finite number of rational operations over the finite languages
$X = \{x_1,\ldots,x_n\}$ and $\{w_1,\ldots,w_k\}$. By Theorem \ref{mainth},
this implies that a finitely presented monomial algebra $A = F/I$ has always
a rational multigraded Hilbert series. Observe that, for the graded case, this agrees
with the classical results in \cite{Go} and \cite{Uf1}. Since Theorem \ref{mainth}
and the corresponding method are more general than the finitely presented case,
in Section \ref{example} we will illustrate the computation of the multigraded
Hilbert series of an infinitely presented but regular monomial algebra.
Note that regular monomial algebras are also called {\em automaton monomial
algebras} \cite{Uf2}.

\section{Finite-dimensional case and truncation}
\label{truncation}

Let $C$ be a regular monomial cyclic right module. In this section
we aim to characterize the case in which $C$ is finite-dimensional. Moreover,
for an infinite-dimensional $C$, we want to develop methods for the polynomial
approximation of the multigraded Hilbert series of $C$, which is possibly
not rational and unknown. To develop the theory for the finite-dimensional case,
it is convenient to switch to the (graded) Hilbert series $\HS'(C)$ that is
obtained from $\HS(C)$ by identifying all variables $t_i$ ($1\leq i\leq n$)
with a single variable $t$. Let $C = F/I$ and consider the orbit
$\O_I = \{I_1,\ldots,I_r\}$ ($I_1 = I$). Moreover, let $A_I$ and $\C_I$ be the
corresponding adjacency matrix and constant vector. As in Theorem \ref{mainth},
we have that the Hilbert series $\HS'(C_k)$ ($C_k = F/I_k$) are obtained
by solving the matrix equation
\[
(E_r - t\cdot A_I)\vH' = \C_I
\]
where $\vH' = (\HS'(C_1),\ldots,\HS'(C_r))^t$. We want to understand when
the (regular monomial cyclic) right modules $C_k$ have finite dimensions, that is,
the solution vector $\vH'$ has all entries in the polynomial algebra $\QQ[t]$.
For any monomial $w\in W$, we have already observed that $w\in I$ if and only
if $T_w(I) = \langle 1 \rangle$ which implies that $\langle 1 \rangle\in\O_I$.
Assume now that $C\neq 0$, that is, $I\neq \langle 1\rangle$ and hence
$r > 1$. Assume that $I_r = \langle 1 \rangle$ and define the {\em reduced
orbit of $I$} as
\[
\bO_I = \O_I\setminus\{\langle 1\rangle\} = \{I_1,\ldots,I_{r-1}\}.
\]
Clearly, $\C_I = (1,\ldots,1,0)^t$ and $\HS'(C_r) = 0$ and therefore
we consider $\bvH' = (\HS'(C_1),\ldots,\HS'(C_{r-1}))^t$ and $\bC_I =
(1,\ldots,1)^t$. Moreover, we denote by $\bA_I$ the square matrix
that is obtained from $A_I$ by deleting the $r$-th row and the $r$-th column.
Finally, we put $\bp_I(t) = \det(t\cdot E_{r-1} - \bA_I)$, that is, $\bp_I(t)$ is
the characteristic polynomial of $\bA_I$. We call $\bA_I,\bp(t)$ and $\bC_I$
respectively, the {\em reduced adjacency matrix, reduced characteristic polynomial
and reduced constant vector of $I$}.

Since $\HS'(C_r) = 0$, it is clear that one can obtain the Hilbert series
$\HS'(C_k)$ ($1\leq k\leq r-1$) by solving the reduced matrix equation
\begin{equation}
\label{redmateq}
(E_{r-1} - t \cdot \bA_I) \bvH' = \bC_I.
\end{equation}
Note that if we assume $\det(E_{r-1} - t \cdot \bA_I)\in\QQ\setminus\{0\}$,
then the matrix inverse $(E_{r-1} - t \cdot \bA_I)^{-1}$ belongs to
$\QQ[t]^{(r-1)\times (r-1)}$ and hence all solutions of equation (\ref{redmateq})
are in fact in the polynomial algebra $\QQ[t]$. Moreover, since
$\det(E_{r-1} - t \cdot \bA_I) = t^{r-1}\bp_I(1/t)\neq 0$, we have in particular
that $\det(E_{r-1} - t \cdot \bA_I) = 1$ if and only if the matrix $\bA_I$
is nilpotent, that is, $\bp_I(t) = t^{r-1}$. In other words, the latter condition
implies that the right module $C$ (in fact each $C_k$) is finite-dimensional.
We now show that the nilpotency of $\bA_I$ is also a necessary condition.

\begin{theorem}
\label{fdprop}
Let $C = F/I$ be a finite-dimensional monomial cyclic right module.
It holds that $C$ is a regular module. Moreover, if $C\neq 0$ and
$\bO_I = \{I_1,\ldots,I_{r-1}\}$ is the reduced orbit of $I$, then
$T_{x_i}(I_k) = I_l$ $(1\leq i\leq n, 1\leq k,l\leq r-1)$
implies that $\dim C_k > \dim C_l$. 
\end{theorem}

\begin{proof}
Consider the key short exact sequence of Theorem \ref{mainth}, namely
\[
0\to \bigoplus_{i=1}^n C_{x_i}\to C\to C/B\to 0
\]
where $C_{x_i} = F/I_{x_i}$ and $I_{x_i} = T_{x_i}(I) = (I :_R x_i)$.
Since $\dim C < \infty$, one has that all (monomial cyclic) right modules
$C_{x_i}$ are also finite-dimensional and
\[
\dim C = \sum_{i=1}^n \dim C_{x_i} + 1.
\]
We conclude that $\dim C > \dim C_{x_i}\geq 0$, for all $i = 1,2,\ldots,n$.
By iterating the above argument along the orbit of $I$, we obtain that
this is finite. In fact, at each iteration, for a non-zero right module $C_k = F/I_k$
one has that the condition $T_{x_i}(I_k) = I_l$ implies that $\dim C_k > \dim C_l$.
\end{proof}

\begin{theorem}
Let $C = F/I$ be a non-zero regular monomial cyclic right module and consider
the corresponding reduced adjacency matrix $\bA_I$. Then, $C$ is
finite-dimensional if and only if $\bA_I$ is a nilpotent matrix.
\end{theorem}

\begin{proof}
By the arguments at the beginning of this section, it remains to prove the necessary
condition and hence let us assume that $\dim C < \infty$. By Theorem \ref{fdprop},
we have that $\dim C_k < \infty$ where $C_k = F/I_k$ and $\bO_I =
\{I_1,\ldots,I_{r-1}\}$. Assume now that the reduced orbit $\bO_I$ is ordered
according to the dimensions, namely $\dim C_k > \dim C_l$ implies that $k < l$.
Again by Theorem \ref{fdprop}, we obtain that the matrix $\bA_I$ is strictly upper
triangular and hence nilpotent.
\end{proof}

For the finite-dimensional case, we remark that the strictly upper triangular structure
of the reduced adjacency matrix makes the computation of the Hilbert series (actually,
polynomial) a fast one. In fact, in Section \ref{timings} we will present
some tests where this computation performs better than normal words enumeration.
This is of course true also for multigraded Hilbert series since $\bA_I$ is
a strictly upper triangular matrix if and only if so are the matrices $\bA_I^{(i)}$
($1\leq i\leq n$), where by definition $\bA_I = \bA_I^{(1)} + \cdots + \bA_I^{(n)}$.
Finally, it is important to mention that Ufnarovski \cite{Uf1} also provided
a graph theoretic characterization of the finite-dimensionality of a finitely
presented monomial algebra.

\medskip
We now consider the problem of computing some truncation of a (multigraded)
Hilbert series. In fact, solving this problem may be of interest if the sum
of the series is difficult to determine because, for instance, it is not a rational
function. Moreover, if an algebra $A = F/I$ is invariant under the action
of the general linear group $\GL_n(\KK)$, the multigraded Hilbert
series of $A$ is a symmetric function. Hence, one may want to decompose
this function as a sum of Schur (polynomial) functions because this provides
the $\GL_n(\KK)$-module structure of $A$. We will give more details about
this application of the truncation by studying a concrete example in Section
\ref{example}.

For the computation of Hilbert series, on account of Theorem \ref{macth}
and Proposition \ref{cycdec}, we are reduced to consider the monomial cyclic case.
Since the main applications concern algebras, let us consider a monomial algebra
$A = F/I$, that is, $I\subset F$ is a monomial two-sided ideal of $F$.
We assume here the standard multigrading for $F$ and $A$, in order to simplify
the notation. Consider the (monomial) two-sided ideal $B =
\langle x_1,\ldots,x_n \rangle$ and its power $B^{d+1}$ ($d\geq 0$) that is generated
by all monomials $w\in W$ such that $\tdeg(w) = d+1$. In other words, one has that
$B^{d+1} = \sum_{k\geq d+1} F_k$. We consider the finite-dimensional monomial algebra
$A^{(d)} = F/I^{(d)}$ where $I^{(d)} = I + B^{d+1}$ and we call this algebra
the {\em $d$-th truncation of $A$}. In fact, it is clear that $A^{(d)}$ is isomorphic
to the vector space $\bigoplus_{k\leq d} A_k$ and hence the multigraded Hilbert series
(in fact, polynomial) $\HS(A^{(d)})$ is the truncation at total degree $d$
of the multigraded Hilbert series $\HS(A)$, that is
\[
\HS(A^{(d)}) = \sum_{|\balpha|\leq d} \dim A_\balpha\, t_1^{\alpha_1}\cdots t_n^{\alpha_n}.
\]
Since the computation of the Hilbert series is based on the computation of the colon
right ideals, we now analize this operation for the two-sided ideals $I^{(d)}$.
We show that one can compute $T_w(I^{(d)}) = (I^{(d)} :_R w)$ ($w\in W$)
without having to deal with the $n^{d+1}$ monomial generators of the ideal $B^{d+1}$.

\begin{lemma}
\label{lem1}
Let $I,J$ be monomial two-sided ideals and $w\in W$. Then $(I + J :_R w) =
(I :_R w) + (J :_R w)$.
\end{lemma}

\begin{proof}
Consider the monomial basis $\{v_i\}\cup\{w_j\}$ of the ideal $I + J$, where
$\{v_i\}$ is a monomial basis of $I$ and $\{w_j\}$ is a monomial basis of $J$.
With the notation of Proposition \ref{colbas}, we immediately have that
\[
\begin{array}{l}
(I + J :_R w) = \sum_i R(w,v_i) + \sum_j R(w,w_j) + I + J \\
\quad = (I :_R w) + (J :_R w).
\end{array}
\]
\end{proof}

\begin{lemma}
\label{lem2}
If $d' = \tdeg(w)$ $(w\in W)$ then
\[
(B^d :_R w) =
\left\{
\begin{array}{cl}
\langle 1 \rangle & \mbox{if}\ d'\geq d, \\
B^{d - d'} & \mbox{otherwise}.
\end{array}
\right.
\]
\end{lemma}

\begin{proof}
If $d'\geq d$ then clearly $w\in B^d$ and therefore $(B^d :_R w) = \langle 1\rangle$.
Assume now $d' < d$. For any $v\in W$ such that $v\in B^{d-d'}$, that is,
$\tdeg(v)\geq d - d'$ one has that $w v\in B^d$ and hence $v\in (B^d :_R w)$.
Conversely, if $v\in (B^d :_R w)$ ($v\in W$) then $w v\in B^d$ and therefore
$d' + \tdeg(v)\geq d$, that is, $v\in B^{d-d'}$.
\end{proof}

\begin{proposition}
\label{trunc}
Let $I\subset F$ be a two-sided monomial ideal and $w\in W$. Put $d' = \tdeg(w)$.
It holds that
\[
(I^{(d)} :_R w) =
\left\{
\begin{array}{cl}
\langle 1 \rangle & \mbox{if}\ d' > d, \\
(I :_R w)^{(d-d')} & \mbox{otherwise}.
\end{array}
\right.
\]
\end{proposition}

\begin{proof}
By assuming $d' > d$ we have that $w\in B^{d+1}$ and hence $w\in I^{(d)} = I + B^{d+1}$,
that is, $(I^{(d)} :_R w) = \langle 1\rangle$. Otherwise, if $d'\leq d$ then
from Lemmas \ref{lem1} and \ref{lem2} it follows that 
\[
\begin{array}{l}
(I^{(d)} :_R w) = (I :_R w) + (B^{d+1} :_R w) = (I :_R w) + B^{d-d'+1} \\
\quad = (I :_R w)^{(d-d')}.
\end{array}
\]
\end{proof}

The results above show that, to compute the finite reduced orbit $\bO_{I^{(d)}}$,
one has simply to compute (according to Proposition \ref{colbas}) the monomial
generators of the colon right ideals $(I :_R w)$ ($\tdeg(w) = d'\leq d$) up to
the total degree $d - d'$. In other words, the $n^{d+1}$ monomial generators
of $B^{d+1}$ are not involved at all in these computations. Moreover, solving
the corresponding matrix equation to obtain the Hilbert series is very efficient,
because we are in the finite-dimensional (strictly upper triangular) case.

\section{An illustrative example}
\label{example}

In this section, by means of a concrete example, we show how the proposed method
for multigraded Hilbert series can be applied to study in an effective way,
finitely generated algebras that are invariant under the action of the general
linear group. To begin with, we introduce some general notions and results about
the action of $G = \GL_n(\KK)$ on the free associative algebra $F = \KK\langle X\rangle$
where $X = \{x_1,\ldots,x_n\}$. For a complete reference we refer to the monographs
\cite{Dr,Fu}. Let $g = (g_{ij})$ be any matrix of the group $G$
and define the algebra automorphism $\rho_g:F\to F$ such that
$x_i\mapsto \sum_j g_{ij} x_j$. Clearly, $\rho_{gh} = \rho_h\rho_g$ for all
$g,h\in G$, that is, one has a right action of $G$ on $F$.
A subspace $V\subset F$ such that $\rho_g(V)\subset V$ for all $g\in G$
is called a {\em $G$-submodule of $F$}. The corresponding (anti)homomorphism
$G\to \Aut_\KK(V)$ is called a {\em polynomial representation of $G$}.
If $F = \bigoplus_d F_d$ is the decomposition of the algebra $F$ in its
homogeneous components, it is clear that each $F_d$ is a $G$-submodule
of $F$. By definition, a subspace $V\subset F$ is {\em graded}
if $V = \sum_d V_d$ where $V_d = V\cap F_d$. In a similar way, one defines
also {\em multigraded} subspaces. Clearly, for graded and multigraded
subspaces we can consider the corresponding Hilbert series as the generating
series of the dimensions of their homomogeneous and multihomogeneous components,
respectively. A $G$-submodule $V\subset F$ is said {\em simple} if there is
no $G$-submodule of $V$ other than 0 and $V$.

Assume now that $\char(\KK) = 0$. By the ``Vandermonde argument'' one has that
all $G$-submodules $V\subset F$ are in fact graded subspaces. Note that each
homogeneous component $V_d = V\cap F_d$ is clearly a $G$-submodule. By Schur's
theory on polynomial representations, all $G$-submodules of $F$ are {\em semisimple},
that is, they are direct sum of simple $G$-submodules. Moreover, a complete
set of (non-isomorphic) simple $G$-submodules $\{W^\lambda\}_\lambda$ is
parametrized by integer partitions $\lambda = (\lambda_1,\ldots,\lambda_k)$
such that $k\leq n$. By denoting $|\lambda| = \sum_i \lambda_i$, one has
in particular that the set $\{W^\lambda\}_{|\lambda|=d}$ occurs in
the decomposition of $F_d$. Each $W^\lambda$ with $|\lambda| = d$ is actually
a multigraded subspace of $F_d$ and one defines its multigraded
Hilbert series
\[
\S_\lambda = \HS(W^\lambda) =
\sum_{|\balpha|=d} \dim W^\lambda_\balpha\, t_1^{\alpha_1}\cdots t_n^{\alpha_n}.
\]
These polynomials are called {\em Schur functions}. Every $\S_\lambda$ is a
symmetric polynomial which can be computed by well-known formulas (see, for instance,
\cite{Fu}). It follows that any $G$-submodule $V\subset F$ is in fact multigraded
and its multigraded Hilbert series $\HS(V) =
\sum_{\balpha\in\NN^n} \dim V_\balpha\, t_1^{\alpha_1}\cdots t_n^{\alpha_n}$
is also a symmetric function which is a (possibly infinite) linear combination
of Schur functions. Precisely, let $V = \bigoplus_\lambda m_\lambda W^\lambda$
be the decomposition of $V$ in its simple $G$-submodules, where the integer
$m_\lambda\geq 0$ denotes the number of times (multiplicity) that a simple
$G$-submodule isomorphic to $W^\lambda$ occurs in $V$. Thus, for the multigraded
Hilbert series we clearly have that $\HS(V) = \sum_\lambda m_\lambda \S_\lambda$.
In other words, it is sufficient to decompose $\HS(V)$ in terms of Schur functions
to have a complete description of how the general linear group $G$ acts
on $V$. When the $G$-submodule $V\subset F$ is infinite-dimensional, the symmetric
function $\HS(V)$ may not be rational and even in the rational case one has
the problem that the Schur function decomposition is indeed a series.
There are some ``nice rational symmetric functions'' that allow the computation
of this decomposition in the infinite-dimensional case. We refer to \cite{BBDGK}
for details about such methods. When $V$ is finite-dimensional, and hence $\HS(V)$
is simply a symmetric polynomial, there are fast algorithms that compute
the Schur function decomposition. These procedures are implemented, for instance,
in the \maple package \textsf{SF} or in the C library \textsf{SYMMETRICA}
which are freely distributed over the Internet \cite{LA,St}. We will make use
of these algorithms to perform the Schur function decomposition in our example.

Consider now a finitely generated algebra $A = F/J$, where the two-sided ideal
$J\subset F$ is also a $G$-submodule. Clearly, the group $G$ acts on $A$ as well
and by the semisemplicity of $F$ we have that $A$ is isomorphic to
a $G$-submodule of $F$. Then, $A$ is multigraded and semisimple and the
multigraded Hilbert series $\HS(A)$ is a linear combination of Schur functions
according to the $G$-module structure of $A$. In particular, the symmetric function
$\HS(A)$ is a polynomial when $A$ is finite-dimensional.

An important class of algebras that are invariant under the action of the general
linear group are the following ones. Let $R$ be an associative algebra and consider
the two-sided ideal
\[
T(R) = \{f\in F\mid f(r_1,\ldots,r_n) = 0,\ \mbox{for all}\ r_1,\ldots,r_n\in R\}.
\]
It is clear that $T(R)\subset F$ is a {\em T-ideal}, that is, it is invariant under
all algebra endomorphisms of $F$ and in particular $T(R)$ is a $G$-submodule.
The finitely generated algebra $A = F/T(R)$ is called the {\em relatively free algebra
in $n$ variables that is defined by $R$}. It is known that the multigraded Hilbert
series of these algebras are rational symmetric functions and the computation of their
Schur function decomposition is very important in the study of a {\em PI-algebra},
that is, an algebra $R$ such that $T(R)\neq 0$. As a reference for PI-theory,
we suggest the books \cite{Dr,GZ}.

Let $E = E(V) = \bigwedge(V)$ be the Grassmann (or exterior) algebra over a vector
space $V$ of countable dimension. Moreover, denote $[f,g] = f g - g f$, for any
$f,g\in F$. For $\char(\KK) = 0$, Latyshev first proved in \cite{La} that
the two-sided ideal $T(E)\subset F$ is generated by the following set of polynomials
\[
[[x_i,x_j],x_k]\, ,\, [x_i,x_j][x_k,x_l] + [x_i,x_k][x_j,x_l]
\]
for all $x_i,x_j,x_k,x_l\in X$. In our example, we will consider the relatively
free algebra $A = F/T(E)$. Note that the $G$-module structure of $A$
was essentially obtained by Krakowski and Regev in \cite{KR}. However, we want to show
here how such a structure can be studied in an algorithmic way by combining our method
for multigraded Hilbert series with algorithms for computing the Schur function
decomposition. Our aim is both to illustrate our method and to suggest
that other interesting $G$-invariant algebras may be investigated in a similar way.
In fact, in Section \ref{timings} we will present more computations of this kind
as a test set.

Fix $n = 3$ and denote $X = \{x,y,z\}$. By Theorem \ref{macth}, a first
step to obtain the multigraded Hilbert series $\HS(A)$ consists in computing
the leading monomial ideal $I = \LM(T(E))$. With respect to the graded left
lexicographic ordering with $x\succ y\succ z$, the two-sided ideal $I$ is minimally
generated by the following infinite set of monomials
\begin{equation*}
\begin{gathered}
x^2y, x^2z, xy^2, xyz, xzy, xz^2, y^2z, yz^2, \\
xyxy, xyxz, xzxy, xzxz, yzyz, \\
yzy^dxy, yzy^dxz\ (d\geq 0).
\end{gathered}
\end{equation*}
This generating set was found in \cite{DLS}, where in fact a minimal \Gr\ basis
of $T(E)$ was given for any number of variables. Clearly, the set of monomials
$I\cap W$ ($W = X^*$) can be obtained by applying a finite number of rational
operations over finite languages, namely
\begin{equation*}
\begin{gathered}
I\cap W = 
W\cdot (\{x^2y, x^2z, xy^2, xyz, xzy, xz^2, y^2z, yz^2, xyxy, xyxz, \\
xzxy, xzxz, yzyz\} \cup \{yz\}\cdot \{y\}^*\cdot \{xy,xz\} )\cdot W.
\end{gathered}
\end{equation*}
By Definition \ref{reglangdef}, we have that $I\cap W$ is a regular language.
By Theorem \ref{reglang} we conclude that $I$ is a regular ideal, that is,
the orbit $\O_I$ is a finite set and therefore the multigraded Hilbert series
$\HS(A)$ is a rational function. Note that this function is also a symmetric one
because $T(E)$ is a $G$-submodule of $F$. Recall that to describe the orbit $\O_I$
and the linear equations relating all the series corresponding to the monomial right
ideals in $\O_I$, one has to compute the colon right ideals $I_w = (I :_R w)$
($w\in W$). Since $I$ is a two-sided ideal, one may use Proposition \ref{colbas}
for that purpose. For instance, to compute $I_x = (I :_R x)$ one considers
the following right ideals
\begin{equation*}
\begin{gathered}
R(x,x^2y) = \langle xy \rangle,
R(x,x^2z) = \langle xz \rangle,
R(x,xy^2) = \langle y^2 \rangle,
R(x,xyz) = \langle yz \rangle, \\
R(x,xzy) = \langle zy \rangle,
R(x,xz^2) = \langle z^2 \rangle,
R(x,y^2z) = 0,
R(x,yz^2) = 0, \\
R(x,xyxy) = \langle yxy \rangle,
R(x,xyxz) = \langle yxz \rangle,
R(x,xzxy) = \langle zxy \rangle, \\
R(x,xzxz) = \langle zxz \rangle,
R(x,yzyz) = 0, \\
R(x,yzy^dxy) = 0,
R(x,yzy^dxz) = 0 \ (d\geq 0).
\end{gathered}
\end{equation*}
We conclude that $I_x = \langle xy, xz, y^2, yz, zy, z^2, yxy, yxz, zxy, zxz \rangle + I$.
In a similar way, one obtains
\begin{equation*}
\begin{gathered}
R(y,x^2y) = 0,
R(y,x^2z) = 0,
R(y,xy^2) = 0,
R(y,xyz) = 0, \\
R(y,xzy) = 0,
R(y,xz^2) = 0,
R(y,y^2z) = \langle yz \rangle,
R(y,yz^2) = \langle z^2 \rangle, \\
R(y,xyxy) = 0,
R(y,xyxz) = 0,
R(y,xzxy) = 0, \\
R(y,xzxz) = 0,
R(y,yzyz) = \langle zyz \rangle, \\
R(y,yzy^dxy) = \langle zy^dxy \rangle,
R(y,yzy^dxz) = \langle zy^dxz \rangle\ (d\geq 0)
\end{gathered}
\end{equation*}
and hence $I_y = \langle yz, z^2, zyz \rangle + \langle zy^dxy, zy^dxz\mid d\geq 0 \rangle + I$.
Moreover, it holds immediately that $I_z = I$. By denoting $C_x = F/I_x, C_y = F/I_y$
the monomial cyclic right modules corresponding to the monomial right ideals
$I_x, I_y$, we obtain the first linear equation
\begin{equation}
\HS(A) = t_1 \HS(C_x) + t_2 \HS(C_y) + t_3 \HS(A) + 1.
\end{equation}
Since $I_x,I_y\in\O_I$, we now have to compute the right colon ideals
$I_{x^2},I_{xy},I_{xz}$ and $I_{yx},I_{y^2},I_{yz}$. The following equalities hold
\begin{equation*}
\begin{gathered}
R(x^2,x^2y) = \langle y, xy \rangle,
R(x^2,x^2z) = \langle z, xz \rangle,
R(x^2,xy^2) = \langle y^2 \rangle,
R(x^2,xyz) = \langle yz \rangle, \\
R(x^2,xzy) = \langle zy \rangle,
R(x^2,xz^2) = \langle z^2 \rangle,
R(x^2,y^2z) = 0,
R(x^2,yz^2) = 0, \\
R(x^2,xyxy) = \langle yxy \rangle,
R(x^2,xyxz) = \langle yxz \rangle,
R(x^2,xzxy) = \langle zxy \rangle, \\
R(x^2,xzxz) = \langle zxz \rangle,
R(x^2,yzyz) = 0, \\
R(x^2,yzy^dxy) = 0,
R(x^2,yzy^dxz) = 0\ (d\geq 0).
\end{gathered}
\end{equation*}
We have therefore that $I_{x^2} = \langle y, z, xy, xz \rangle + I$. In a similar way,
we compute that $I_{x^2} = I_{xy} = I_{xz}$ and one obtains the equation
\begin{equation}
\HS(C_x) = t_1 \HS(C_{x^2}) + t_2 \HS(C_{x^2}) + t_3 \HS(C_{x^2}) + 1.
\end{equation}
It is easy to check that $I_{yx} = I_x$ and $I_{y^2} = \langle z, yz \rangle + I$.
Moreover, we compute that
$I_{yz} = \langle z, yz \rangle + \langle y^dxy, y^dxz\mid d\geq 0 \rangle + I$
and one obtains the equation
\begin{equation}
\HS(C_y) = t_1 \HS(C_x) + t_2 \HS(C_{y^2}) + t_3 \HS(C_{yz}) + 1.
\end{equation}
Then, we have that $I_{x^2},I_{y^2},I_{yz}\in\O_I$ and one has to compute the corresponding
colon right ideals $I_{x^3},I_{x^2y},I_{x^2z},I_{y^2x},I_{y^3},I_{y^2z}$ and
$I_{yzx},I_{yzy},I_{yz^2}$. By similar computations, one obtains $I_{x^3} = I_{x^2}$.
Moreover, from $x^2y, x^2z\in I$ it follows immediately that $I_{x^2y},I_{x^2z} = \langle 1\rangle$.
A new equation is hence the following one
\begin{equation}
\HS(C_{x^2}) = t_1 \HS(C_{x^2}) + t_2 \HS(C_{x^2y}) + t_3 \HS(C_{x^2y}) + 1.
\end{equation}
We have the following identities $I_{y^2x} = I_x, I_{y^3} = I_{y^2}, I_{y^2z} = I_{x^2y}$
which imply the equation
\begin{equation}
\HS(C_{y^2}) = t_1 \HS(C_x) + t_2 \HS(C_{y^2}) + t_3 \HS(C_{x^2y}) + 1.
\end{equation}
One has also the identities $I_{yzx} = I_{x^2}, I_{yzy} = I_{yz}, I_{yz^2} = I_{x^2y}$
and the equation
\begin{equation}
\HS(C_{yz}) = t_1 \HS(C_{x^2}) + t_2 \HS(C_{yz}) + t_3 \HS(C_{x^2y}) + 1.
\end{equation}
Finally, we have that $I_{x^2y} = \langle 1 \rangle\in\O_I$ where clearly $\HS(C_{x^2y}) = 0$.
This can be also obtained by the obvious identities $I_{x^2yx} = I_{x^2y^2} = I_{x^2yz} =
I_{x^2y}$ and by the corresponding linear equation (with $c(I_{x^2y}) = 0$)
\begin{equation}
\HS(C_{x^2y}) = t_1 \HS(C_{x^2y}) + t_2 \HS(C_{x^2y}) + t_3 \HS(C_{x^2y}).
\end{equation}
We conclude that $\O_I = \{I,I_x,I_y,I_{x^2},I_{y^2},I_{yz},I_{x^2y}\}$. By solving
the (non-singular) system of the obtained linear equations, one computes the multigraded
Hilbert series of all monomial cyclic right modules corresponding to the elements
of the orbit $\O_I$. In particular, we obtain that
\[
\HS(A) = \frac{t_1t_2 + t_1t_3 + t_2t_3 + 1}{(1 - t_1)(1 - t_2)(1 - t_3)}. \\
\]
Note that this formula is a special case, for $n = 3$, of the general formula (see \cite{BR})
for the multigraded Hilbert series of the relatively free algebra in $n$ variables
of the Grassmann algebra, which is
\[
\frac{\displaystyle \prod_{i=1}^n(1 + t_i) + \prod_{i=1}^n(1 - t_i) }
{\displaystyle 2\cdot \prod_{i=1}^n(1 - t_i)}.
\]

We now consider the problem of understanding the Schur function decomposition
of $\HS(A)$. Recall that $A = F/T(E)$ and $I = \LM(T(E))$ and consider the corresponding
monomial algebra $A' = F/I$. Our previous computations have been made according to
Theorem \ref{macth} which yields that $\HS(A') = \HS(A)$. To obtain information about
the Schur function decomposition of $\HS(A)$, one possible approach consists
in truncating $A'$ at a sufficiently high total degree $d$ to see which decomposition
holds up to that degree. Precisely, with the notation of Section \ref{truncation},
we consider the finite-dimensional monomial algebra $A'^{(d)} = F/I^{(d)}$ which is
the $d$-th truncation of $A'$. We obtain that $\HS(A'^{(d)})$ is the truncation
at the total degree $d$ of the multigraded Hilbert series $\HS(A') = \HS(A)$
which implies that $\HS(A'^{(d)})$ is a symmetric polynomial. Then, the Schur function
decomposition of $\HS(A'^{(d)})$ can be obtained by means of efficient
algorithms and this decomposition corresponds to the $G$-module structure
of $A$ up to the fixed degree $d$.

To show how feasible this method could be for obtaining a sufficiently large truncation
of the series $\HS(A) = \sum_\lambda m_\lambda \S_\lambda$, in our example we fix
$d = 10$. In this case, the orbit $\O_{I^{(10)}}$ consists of 51 ideals and the
corresponding polynomial $\HS(A'^{(10)})$ has 286 monomials in the variables
$t_1,t_2,t_3$. This symmetric function is computed by our methods in 33 milliseconds
on a Linux server (see Section \ref{timings} for the specification of the machine).
Moreover, the Schur function decomposition of $\HS(A'^{(10)})$ takes 0.38 seconds in the
\magma interface of the library \textsf{SYMMETRICA}. These computations return
the decomposition
\[
\HS(A'^{(10)}) = \sum_{0\leq k\leq 10}
\sum_{
\begin{array}{c}
\scriptstyle
p+q = k, \\
\scriptstyle
0\leq q\leq 2
\end{array}
}
\S_{(p,1^q)}
\]
where $(p,1^q)$ reads $(p,1,\ldots,1)$ and we agree that $S_{(0,0)} = 1$.
This suggests the following formula, which was obtained in \cite{KR}
for an arbitrary number $n$ of variables
\[
\HS(A) = \sum_{k\geq 0}
\sum_{
\begin{array}{c}
\scriptstyle
p+q = k, \\
\scriptstyle
0\leq q\leq n-1
\end{array}
}
\S_{(p,1^q)}.
\]
Observe that, in this example, we have computed the rational form of the complete
Hilbert series $\HS(A)$ and therefore $\HS(A'^{(10)})$ could also be obtained
by means of a Taylor expansion. Nevertheless, since we are in the multivariate case,
even this task may be a non-trivial one (recall that $\HS(A'^{(10)})$ has 286 monomials).
Moreover, we remark that for a non-regular monomial algebra, the sum of the Hilbert series
may be completely unknown.

\section{Implementation and Timings}
\label{timings}

In this section, we provide the practical performance of the proposed algorithms
for the computation of graded and multigraded Hilbert series, both in the complete
and in the truncated case. All tests are performed by means of an implementation
that we have developed in the kernel of the computer algebra system \singular
\cite{DGPS}. Recall that the Hilbert series of any graded or multigraded algebra
$A = F/J$ ($F = \KK\langle x_1,\ldots,x_n \rangle$) is the same as the series
of the corresponding monomial algebra $A' = F/I$ where $I = \LM(J)$ is the leading
monomial ideal of $J$. Then, for computing the Hilbert series, our kernel
implementation requires as input a set of monomial generators of $I$.
Precisely, user has to compute a \Gr\ basis of $J$ with respect to some monomial
ordering and to input this basis into the procedure \textsf{nchilb}
which has been implemented in \singular{}'s interpreted language. This function
collects the leading monomials of the \Gr\ basis and converts them in a suitable
format for the kernel's code. See our \singular library \textsf{ncHilb.lib}
for detailed instructions to use it. Note that internally to the kernel,
the implementation uses indeed commutative analogues of noncommutative
monomial ideals according to the ``letterplace correspondence''.
For more details about the wide scope of letterplace methods we refer to
\cite{LSL1,LSL2,LS1,LS3}.

We display here the pseudo code of the proposed algorithm for multigraded algebras
that runs in the kernel.
\begin{algorithm}[H]
\caption{Multigraded Hilbert series algorithm}
\label{HSalgo}
\begin{algorithmic}[1]
\REQUIRE A basis of a monomial two-sided ideal $I\subset F$.
\ENSURE The multigraded Hilbert series of the monomial algebra $A = F/I$.
\STATE $\O_I:= \{I\}$, $\N:= \{I\}$
\STATE matrix $P = ( p_{ki} ):= 0$, column vector $\C_I = (c_k)^t:= 0$
\WHILE{$\N\neq\emptyset$}
\STATE choose $J\in \N$, $\N:=\N\setminus \{J\}$
\STATE $k:=$ position of $J$ in $\O_I$
\IF{$J\neq \langle 1 \rangle$}
\STATE $c_k = 1$ 
\ENDIF
\FOR{$1\leq i \leq n$}
\STATE compute the colon right ideal $J_{x_i}:= T_{x_i}(J)$
\IF{$ J_{x_i}\notin \O_I$}
\STATE $\N:= \N\cup \{J_{x_i}\}$, $\O_I:= \O_I\cup \{J_{x_i}\}$
\ENDIF
\STATE $p_{ki}:= \text{ position of }J_{x_i}\text{ in } \O_I$
\ENDFOR
\ENDWHILE
\STATE ($r\times r$)-matrix $M := 0$ ($r:=$ size of $\O_I$)
\STATE unit matrix $E_r$
\FOR{$1\leq k \leq r$}
\FOR{$1\leq i \leq n$}
\STATE $M[k,p_{ki}]:= M[k,p_{ki}] + t_i$
\ENDFOR
\ENDFOR
\RETURN $\H_1$, where $\vH = (\H_1,\ldots,\H_r)^t$ is a solution of the matrix equation
$(E_r-M)\vH = \C_I$ over the field $\QQ(t_1,\ldots,t_n)$. 
\end{algorithmic}

\end{algorithm}
If all variables $t_i$ are identified in Step 17 with a single variable $t$, then
the algorithm above computes the graded Hilbert series. User needs to provide
the optional parameter ``2'' for the multigraded Hilbert series since by default
the implementation returns the graded one.

\subsection{Computation of Hilbert series of infinitely generated ideals}
\label{variant}

If $I$ is a regular but infinitely generated monomial ideal, it is clear
that one can provide only the finite set of monomial generators of $I$
up to some fixed total degree $d$. We remark that this situation is different
from the notion of truncation of Section \ref{truncation} where, at least
formally (see Proposition \ref{trunc}), one inputs also all monomials of $F$
of degree $d+1$. Thus, for infinitely generated regular ideals, Algorithm \ref{HSalgo}
essentially guesses the sum of the Hilbert series whose correctness has
to be proved by the handling of regular expressions, as we have done
for the example in Section \ref{example}. In fact, one has a strong indication
of the correctness of the computer calculation when this rational function
stabilizes as the degree bound $d$ increases. To obtain correct guesses,
a modification of the Algorithm \ref{HSalgo} is actually required.
Observe that each time a colon right ideal operator $T_{x_i}$ ($1\leq i\leq n$)
is applied to $I$, one has a complete set of monomial generators
of $I_{x_i} = T_{x_i}(I)$ just up to the total degree $d-1$.
This is because a generator of $I_{x_i}$ of degree $d$ may arise from
a generator of $I$ of degree $d+1$ which has not been included in the input.
By iterating the operators $T_{x_i}$, one has that two right colon ideals
$I_w = T_w(I), I_{w'} = T_{w'}(I)$ ($w,w'\in W$) in the orbit of $I$ can be
only compared by means of their monomial generators up to degree $d - d'$,
where $d' = \max(\tdeg(w),\tdeg(w'))$. If $d$ is a suitable large bound, then
this trick usually provides correct comparisons and hence correct Hilbert series.
To access the variant of the Algorithm \ref{HSalgo} for infinitely generated
ideals, user has to provide an optional parameter ``1'' along with the input.

\subsection{Computation of affine Hilbert series}
There are many interesting (noncommutative) algebras $A = F/J$ which are not graded
ones. Consider, for instance, the group algebras of finitely generated groups.
For these algebras, one still has the notion of affine Hilbert function and affine
Hilbert series which are also called ``growth function and growth series''.
Precisely, consider $F_{\leq d} = \sum_{0\leq k\leq d} F_d$ the subspace of $F$
of all polynomials of total degree $\leq d$. Moreover, denote $I_{\leq d} =
I\cap F_{\leq d}$ and $A_{\leq d} = F_{\leq d}/I_{\leq d}$.
The {\em affine Hilbert function} $\HF_{a}'(A)$ is defined by putting,
for all $d\geq 0$,
\[
\HF'_a(A)(d) = \dim A_{\leq d}.
\]
The corresponding generating series $\HS'_a(A)$ is called the {\em affine Hilbert series}.
Let $\prec$ be a graded monomial ordering of $F$, that is, $\tdeg(w) < \tdeg(w')$
($w,w'\in W$) implies that $w\prec w'$. As usual, one considers the corresponding
monomial algebra $A' = A/I$ where $I = \LM(J)$. By similar arguments to the ones
of Theorem \ref{macth}, one proves that the algebras $A$ and $A'$ share the same
affine Hilbert function and series (see \cite{LS2}). Moreover, we immediately have that
\[
\HS'_a(A') = \sum_{d\geq 0} \left(\sum_{k=0}^d \HF'(A')(k)\right) t^d =
\left(\sum_{k\geq 0}t^k \right)\left(\sum_{d\geq 0}\HF'(A')(d) t^d \right)
\]
and hence $\HS_{a}'(A) = \HS'(A')/(1-t)$. Thus, one can easily obtain the affine Hilbert
series of $A$ by computing the graded Hilbert series of $A'$. For the non-homogeneous
test cases, we provide the computational timings of $\HS'(A')$.  

To show the performance of our implementation, we have carried out the computations on
a Dell PowerEdge R720 with two Intel(R) Xeon(R) CPU E5-2690  @ 2.90GHz, 20 MB Cache,
16 Cores, 32 Threads, 192 GB RAM with a Linux operating system (Gentoo). Besides the
experimental implementation in \cite{LS2}, this is the first implementation
in the kernel of a computer algebra system performing the computation of Hilbert
series of noncommutative algebras in general. To test the performance of our algorithms
and their implementations, we provide data for graded and multigraded Hilbert series,
together with their truncations, for various examples.

In the tables below, we abbreviate milliseconds, seconds, minutes as ms, s, m,
respectively. The symbol $\infty$ indicates that the computation has not been finished
within $1$ hour. The computing times for graded and multigraded Hilbert series are
indicated by $\HS$ and $\mHS$, respectively. We denote by $\#\O_I$ the cardinality
of the orbit $\O_I$. Moreover, we let $\O_I = \{T_{w_1}(I),\ldots,T_{w_r}(I)\}$
($w_i\in W$), where we assume that if $T_{w_i}(I) = T_{w'_i}(I)$ then
$\tdeg(w_i)\leq \tdeg(w'_i)$. We indicate by max$\{\vert w\vert\}$ the maximal
total degree of the words $w_1,\ldots,w_r$ and by $\Sol$ and $\mSol$ the cpu timings
for solving the linear systems over the rational functions fields involved in the
computation of graded and multigraded series, respectively. Since the complexity
of a computation depends also on the cardinality of a (minimal) monomial basis
of $I$ and the maximum total degree in it, we provide these details, as well,
for some examples and we denote them by $\#I$ and $deg(I)$, respectively.
The base field $\KK$ is always assumed to be the field of rational numbers.
  
\subsection{Tests of affine Hilbert series}

\begin{example}
Here we give the computational details for some classes of non-graded algebras.
\begin{compactenum} 
\item Consider the following Coxeter matrices
\[
C_1 =
\begin{bmatrix}
1 & 3 & 3 & 2 & 2 \\
3 & 1 & 3 & 2 & 2\\
3 & 3 & 1 & 3 & 3 \\
2 & 2 & 3 & 1 & 3\\
2 & 2 & 3 & 3 & 1\\
\end{bmatrix}
\,,\,
C_2 =
\begin{bmatrix}
1 & 2 & 3 & 2 & 3 \\
2 & 1 & 3 & 2 & 2\\
3 & 3 & 1 & 3 & 3 \\
2 & 2 & 3 & 1 & 3\\
3 & 2 & 3 & 3 & 1\\
\end{bmatrix}
\]
For a parameter $\delta \neq 0$, we consider the following two-sided ideals of
the free associative algebra $F = \KK\langle x_1,\ldots, x_5 \rangle$
\begin{align*}
J_1 &=\langle (x_{j}-\delta)(x_{j}+1)\,  (1\leq j\leq 5), 
x_1x_2x_1-x_2x_1x_2,
x_1x_3x_1-x_3x_1x_3,
\\&\quad \quad x_1x_4  -x_4x_1,
x_1x_5-x_5x_1,
x_2x_3x_2-x_3x_2x_3,
x_2x_4-x_4x_2,
x_2x_5 \\&\quad \quad -x_5x_2,
x_3x_4x_3-x_4x_3x_4,
x_3x_5x_3-x_5x_3x_5,
x_4x_5x_4-x_5x_4x_5 \rangle ,
\\[4 mm]
J_2 &=\langle (x_{j}-\delta)(x_{j}+1)\,  (1\leq j\leq 5), 
x_1x_2-x_2x_1,
x_1x_3x_1-x_3x_1x_3,
x_1x_4 \\&\quad \quad -x_4x_1,
x_1x_5-x_5x_1,
x_2x_3x_2-x_3x_2x_3,
x_2x_4-x_4x_2,
x_2x_5-x_5x_2,
\\&\quad \quad x_3x_4x_3  -x_4x_3x_4,
x_3x_5x_3-x_5x_3x_5,
x_4x_5x_4-x_5x_4x_5 \rangle.
\end{align*}
Then, the quotient algebras $\text{HA}_1 = F/J_1, \text{HA}_2 = F/J_2$
are by definition the Hecke algebras corresponding to the matrices $C_1, C_2$,
respectively.
\item Again, let us consider the following Coxeter matrices
\[
C_3 =
\begin{bmatrix}
1 & 2 & 2 & 3 & 3 \\
2 & 1 & 3 & 2 & 2\\
2 & 3 & 1 & 3 & 3 \\
3 & 2 & 3 & 1 & 3\\
3 & 2 & 3 & 3 & 1\\
\end{bmatrix}
\,,\,
C_4 =
\begin{bmatrix}
1 & 2 & 2 & 3 & 3 \\
2 & 1 & 2 & 3 & 2\\
2 & 2 & 1 & 3 & 3 \\
3 & 3 & 3 & 1 & 3\\
3 & 2 & 3 & 3 & 1\\
\end{bmatrix}
\]
and the following two-sided ideals of $F$ 
\begin{align*}
J_3 &= \langle (x_{i}x_{j})^{C_3[i,j]}-1\, (1\leq i,j\leq 5) \rangle, \\ 
J_4 &= \langle (x_{i}x_{j})^{C_4[i,j]}-1\, (1\leq i,j\leq 5) \rangle.
\end{align*}
The quotient algebras $\text{CA}_3 = F /J_3, \text{CA}_4 = F/J_4$
are the group algebras of the Coxeter groups defined by the matrices $C_3, C_4$,
respectively.
 
\item Let $F_n = \KK\langle x_1,\ldots,x_{n-1}\rangle$ be the free associative algebra
in $n-1$ variables. For a scalar $\lambda\neq 0$, consider the two-sided ideal
$L_n\subset F_n$ generated by the following relations 
\[
 x_{i}^2         = \lambda x_{i}~ \forall  i, ~~
 x_{i}x_{j}      = x_{j}x_{i} \text{ if } \vert j-i\vert > 1, ~~ 
 x_{i}x_{j}x_{i} = \lambda x_{i} \text{ if } \vert j-i\vert=1. 
\]
Then, the quotient algebra $\text{TL}_n = F_n/L_n$ is by definition a
Temperley-Lieb algebra. We consider $\text{TL}_{11}$ and $\text{TL}_{12}$
for the computation. These are finite dimensional algebras. 
\end{compactenum}
\begin{center}
\emph{
\begin{tabular}{|c|c|c|c|c|c|c|}
\hline
Tests& $\HS$& $\#\O_I$&$\text{max}\{\vert w\vert\}$&$\Sol$& $\#$I & deg(I)\\
\hline
\hline
$\text{HA}_1$ & $\unit[263]{ms} $  &$66$ &$10$ & $\unit[20]{ms}$ & $80$&$14$\\
$\text{HA}_2$ & $\unit[696]{ms} $  &$87$ &$13$ & $\unit[69]{ms}$ &$121$&$17$\\
\hline
$\text{CA}_3$ & $\unit[1.05]{s}$  &$88$ &$12$  & $\unit[64]{ms}$ &$155$&$19$\\
$\text{CA}_4$ & $\unit[1.10]{s}$  &$92$ &$13$  & $\unit[116]{ms}$ &$147$&$21$\\ 
\hline
$\text{TL}_{11}$ & $\unit[3.04]{s}$ & $213$ & $9$ & $\unit[25]{ms}$&$136$  &$11$\\
$\text{TL}_{12}$ & $\unit[7.03]{s}$ & $278$ & $10$ & $\unit[47]{ms}$&$166$  &$12$\\
\hline
\end{tabular}
}
\end{center}
\end{example}

\subsection{Tests of multigraded Hilbert series}
\label{gl_n-tables}

\begin{example}[Relative free algebras] In these examples, we provide the data for
the graded and multigraded Hilbert series of the following $\GL_n(\KK)$-invariant algebras.
\begin{compactenum}
\item The relatively free algebra in $n$ variables that is defined by the Grassmann
(or exterior) algebra $E$. This example has been described in full details in Section
\ref{example}. We denote by {\em rf$\_$ext}n this relatively free algebra.

\item The relatively free algebra in $n$ variables corresponding to the algebra
$\text{UT}_2(\KK)$ of 2 by 2 upper triangular matrices. General results for
$\text{UT}_m(\KK)$, which were obtained by Maltsev \cite{Ma}, imply that
the T-ideal of the polynomial identities satified by $\text{UT}_2(\KK)$
is generated, as a two-sided ideal of $F = \KK\langle x_1,\ldots,x_n\rangle$,
by the following infinite basis
\[
[x_i,x_j] w [x_k,x_l]
\]
where $x_i,x_j,x_k,x_l$ are any variables and $w$ is any word. In the examples 
under consideration, we assume that $w$ is also a variable, that is, we input
the generators of $T(\text{UT}_2(\KK))$ up to the total degree 5. In fact, this is
enough to obtain the correct multigraded Hilbert series by means of the variant
of the Algorithm \ref{HSalgo} described in Subsection \ref{variant}.
We denote this test set as {\em rf$\_$tri}n. For the reader's convenience,
we recall that the formula for the multigraded Hilbert series of the relatively
free algebra in $n$ variables of $\text{UT}_m(\KK)$ is the following one (see \cite{BD})
\[
\sum_{j=1}^{m} \binom{m}{j}
\left(\prod_{i=1}^{n}\frac{1}{1-t_i}\right)^j\left(t_1+\cdots+t_n-1\right)^{j-1}.
\]

\end{compactenum}
\begin{center}
\emph{
\begin{tabular}{|c|c|c|c|c|c|c|c|}
\hline
Tests& $\HS$ & $\mHS$ & $\#\O_I$&$\text{max}\{\vert w\vert\}$& $\mSol$& $\#$I & deg(I)\\
\hline
\hline
rf$\_$ext6 & $\unit[441] {ms}$ & $\unit[654]{ms}$ & $28$ & $4$ & $\unit[215]{ms}$&$715$&$7$\\
rf$\_$ext7 & $\unit[2.03] {s}$ & $\unit[4.36]{s}$ & $39$ & $4$ & $\unit[2.33]{s}$&$1675$&$7$ \\
\hline
rf$\_$tri6 & $\unit[278] {ms}$ & $\unit[1.39] {s}$ &$13$ & $4$& $\unit[1.11] {s}$  &$981$&$5$\\
rf$\_$tri7 & $\unit[903] {ms}$ & $\unit[18.17] {s}$ &$15$ & $4$& $\unit[17.26] {s}$  &$2079$&$5$\\
\hline
\end{tabular}
}
\end{center}
We remark that the obtained multigraded Hilbert series are rational symmetric functions.

\end{example}
\begin{example}[Universal enveloping algebras] We consider the following algebras.
\begin{compactenum}
\item Let $I\subset F=\KK\langle x_1, \ldots,x_n \rangle$ be the two-sided ideal that
is generated by all the commutators $[x_{i_{1}},\ldots,x_{i_{d}}]$ of length $d$.
We have that $F/I$ is the universal enveloping algebra of the free nilpotent of class
$d-1$ Lie algebra with $n$ generators. This is clearly a $\GL_n(\KK)$-invariant algebra
and hence its multigraded Hilbert series is a symmetric function. We denote this algebra
by n{\em unil}d when $F$ is generated by $n$ variables and $I$ is generated by $d$-length
commutators.
\item Another $\GL_n(\KK)$-invariant two-sided ideal $J\subset F$ is generated by
all commutators $[ [x_{i_{1}},x_{i_{2}}], [x_{i_{3}},x_{i_{4}}] ]$ where
$1\leq i_1,i_2,i_3,i_4\leq n$. This defines $F/J$ as the universal enveloping algebra
of the free metabelian Lie algebra with $n$ generators. We denote this example as
{\em umeta}n for the $n$ variables.  
\end{compactenum}
\begin{center}
\emph{
\begin{tabular}{|c|c|c|c|c|c|c|c|}
\hline
Tests& $\HS$ & $\mHS$ & $\#\O_I$&$\text{max}\{\vert w\vert\}$&$\mSol$& $\#$I & deg(I)\\
\hline
\hline
4unil4 & $\unit[345]{ms}$ & $\infty$  &   $77$    & $5$& $\infty$   &$110$&$6$\\
5unil3 & $\unit[20]{ms}$& $\unit[1.84]{m}$    & $26$   & $3$& $\unit[1.84]{m}$ & $50$&$4$\\
\hline
umeta3 & $\unit[<1]{ms}$& $\unit[19]{ms}$  & $7$    & $4$ & $\unit[18]{ms}$  &$3$&$4$\\
umeta4 & $\unit[4]{ms}$& $\unit[3.49]{s}$  & $15$    & $4$& $\unit[3.49]{s}$ &$15$&$4$\\
\hline
\end{tabular}
}
\end{center} 

\end{example}
It can be observed from the two tables above that for multigraded Hilbert series
most of the computing time is spent in solving the linear system. Maybe this is a side effect
of the implementation of multivariate rational function fields in \singular.

\subsection{Tests of truncated Hilbert series}

As we have explained in Section \ref{example}, the motivation to develop an algorithm
to compute truncated Hilbert series begins with $\GL_n(\KK)$-invariant algebras and the necessity
to obtain the Schur function decomposition of the symmetric function which is the sum of the
multigraded series. Since it is difficult to obtain this decomposition when this symmetric
function is not a polynomial, we may decide to truncate the Hilbert series at a sufficiently
high degree and compute the corresponding partial decomposition in order to guess the complete
one. Of course, one may obtain the truncation of a rational Hilbert series as its Taylor
approximation up to some fixed degree but sometimes computing such a series may not be feasible,
mainly because of two reasons:
\begin{compactenum}
\item Fairly often, computing a \Gr\ basis up to an appropriate high degree required
to obtain the correct Hilbert series when the corresponding leading monomial ideal
is an infinitely generated regular ideal, is either too costly or not feasible. Even in the
case of finitely generated leading monomial ideals, computing \Gr\ bases may not be efficient.
\item Solving a linear system in a multivariate rational function field (with rational coefficients)
may also be unfeasible.
\end{compactenum}
In these cases, our implementation of the algorithm for truncated algebras presented in
Section \ref{truncation}, offers a feasible computation of the truncated Hilbert series.
We remark that it computes also such a truncation for non-regular monomial algebras, where
the sum of the complete Hilbert series cannot be obtained, up to now, by general methods.
User needs to provide an optional parameter ``$d+1$''
to the implementation for obtaining the truncated Hilbert series up to the total degree $d$.
If no such parameter is given to \textsf{nchilb}, it computes the complete Hilbert series.

Furthermore, we notice that the main difference of timings between the graded and
multigraded computation is due to the complexity of solving a linear system over multivariate
versus univariate function fields. The timings do not differ much for finite dimensional
algebras and the truncated case. In fact, for these cases the reduced adjacency matrix
is strictly upper triangular and the corresponding linear system is therefore easy to solve.

Now, we provide data of the computation of truncated Hilbert series of the
$\GL_n(\KK)$-invariant algebras that we have considered in Subsection \ref{gl_n-tables}.
We compare our truncation algorithm with an available implementation \textsf{lpHilbert}
in \singular (see library \textsf{fpadim.lib}) that computes truncated graded Hilbert
series by means of normal words enumeration. In the following tables, $\tdeg$ denotes
the truncation degree and $\#\mHS$ indicates the number of monomials in the corresponding
truncated multigraded Hilbert series.

\begin{center}
\begin{tabular}{|c|c|c|c|c|c|c|c|}
\hline
Tests & $\tdeg$ & $\lph$ & $\HS$ & $\mHS$ & $\#\O_I$ & $\mSol$ &$\#\mHS$\\  
\hline
\hline
rf$\_$ext6 &  $7$ & $\unit[56.19]{s}$ & $\unit[3.04]{s}$ & $\unit[3.05]{s}$ & $108$& $\unit[7]{ms}$&$1716$\\
rf$\_$ext7 &  $7$ & $\unit[5.08]{m}$ & $\unit[11.79]{s}$ & $\unit[11.81]{s}$ & $143$& $\unit[15]{ms}$&$3432 $\\
\hline
rf$\_$tri6  & 7  &$\unit[58.13]{m}$   & $\unit[21.19]{s}$  & $\unit[21.20]{s}$  &  $53$ & $\unit[3]{ms}$  & $1716$  \\
rf$\_$tri7  & 7  &$\infty$   & $\unit[1.53]{m}$  & $\unit[1.53]{m}$  &  $61$ & $\unit[6]{ms}$  & $3432$  \\
\hline
5unil5  & $9$&$\infty$& $\unit[1.62]{m}$   & $\unit[1.62]{m}$  & $1270$ & $\unit[511]{ms}$  & $2002$  \\
6unil6  & $7$&$\infty$& $\unit[12.27]{m}$   & $\unit[12.07]{m}$  & $ 492$& $\unit[14]{ms}$  & $1716$  \\
\hline
umeta6 & $10$ & $\infty$ & $\unit[13.41]{s}$ & $\unit[13.49]{s}$ & $365$ & $\unit[90]{ms}$ & $8008$\\
umeta7 & $10$ & $\infty$ & $\unit[2.93]{m}$ & $\unit[2.94]{m}$ & $610$& $\unit[340]{ms}$ & $19448$\\
\hline
\end{tabular}
\end{center}

By comparing the table above with the tables of Subsection \ref{gl_n-tables}, one can
observe that the computation of the complete Hilbert series may sometimes be faster
than computing a truncated one, as for the examples rf$\_$ext$n$ and rf$\_$tri$n$.
In many other cases, one has the opposite situation, as for the examples $n$unil$d$
and umeta$n$, where many complete multigraded Hilbert series cannot be computed
in a reasonable time.

We conclude with the beatiful picture of the Schur function decomposition of one of the
truncated (multigraded) Hilbert series that we have computed.
The following decomposition of $\HS(A^{(10)})$ for $A=$ {\em umeta6} takes 5 minutes
using \magma\\
\begin{small}
$
1
 +
~\S[1]
 +
~\S[1,1] + ~\S[2]
 +
~\S[1,1,1] + 2~\S[2,1] + ~\S[3]
 +
~\S[1,1,1,1] + 2~\S[2,1,1] + 2~\S[2,2] + 3~\S[3,1] + ~\S[4]
 +
~\S[1,1,1,1,1] + 2~\S[2,1,1,1] + 3~\S[2,2,1] + 4~\S[3,1,1] + 5~\S[3,2] + 
4~\S[4,1] + ~\S[5]
+
~\S[1,1,1,1,1,1] + 2~\S[2,1,1,1,1] + 3~\S[2,2,1,1] + 3~\S[2,2,2] + 5~\S[3,1,1,1] 
+ 10~\S[3,2,1] + 5~\S[3,3] + 7~\S[4,1,1] + 9~\S[4,2] + 5~\S[5,1] + ~\S[6]
+
2~\S[2,1,1,1,1,1] + 3~\S[2,2,1,1,1] + 4~\S[2,2,2,1] + 5~\S[3,1,1,1,1] + 
13~\S[3,2,1,1] + 11~\S[3,2,2] + 13~\S[3,3,1] + 10~\S[4,1,1,1] + 23~\S[4,2,1] + 
14~\S[4,3] + 11~\S[5,1,1] + 14~\S[5,2] + 6~\S[6,1] + ~\S[7]
+
3~\S[2,2,1,1,1,1] + 4~\S[2,2,2,1,1] + 3~\S[2,2,2,2] + 5~\S[3,1,1,1,1,1] + 
14~\S[3,2,1,1,1] + 20~\S[3,2,2,1] + 20~\S[3,3,1,1] + 21~\S[3,3,2] + 
11~\S[4,1,1,1,1] + 37~\S[4,2,1,1] + 30~\S[4,2,2] + 45~\S[4,3,1] + 14~\S[4,4] + 
18~\S[5,1,1,1] + 44~\S[5,2,1] + 28~\S[5,3] + 16~\S[6,1,1] + 20~\S[6,2] + 
7~\S[7,1] + ~\S[8]
+
4~\S[2,2,2,1,1,1] + 4~\S[2,2,2,2,1] + 14~\S[3,2,1,1,1,1] + 25~\S[3,2,2,1,1] + 
17~\S[3,2,2,2] + 24~\S[3,3,1,1,1] + 48~\S[3,3,2,1] + 19~\S[3,3,3] + 
12~\S[4,1,1,1,1,1] + 46~\S[4,2,1,1,1] + 70~\S[4,2,2,1] + 84~\S[4,3,1,1] + 
86~\S[4,3,2] + 54~\S[4,4,1] + 23~\S[5,1,1,1,1] + 84~\S[5,2,1,1] + \\
67~\S[5,2,2] +
108~\S[5,3,1] + 42~\S[5,4] + 30~\S[6,1,1,1] + 75~\S[6,2,1] + 48~\S[6,3] + 
22~\S[7,1,1] + 27~\S[7,2] + 8~\S[8,1] + ~\S[9]
+
4~\S[2,2,2,2,1,1] + 3~\S[2,2,2,2,2] + 26~\S[3,2,2,1,1,1] + 30~\S[3,2,2,2,1] + 
25~\S[3,3,1,1,1,1] + 69~\S[3,3,2,1,1] + \\
50~\S[3,3,2,2] + 55~\S[3,3,3,1] + 
50~\S[4,2,1,1,1,1] + 103~\S[4,2,2,1,1] + 69~\S[4,2,2,2] + 117~\S[4,3,1,1,1] + 
238~\S[4,3,2,1] + 94~\S[4,3,3] + 117~\S[4,4,1,1] + 127~\S[4,4,2] + 
27~\S[5,1,1,1,1,1] + 120~\S[5,2,1,1,1] + 186~\S[5,2,2,1] + 237~\S[5,3,1,1] + 
238~\S[5,3,2] + 190~\S[5,4,1] + 42~\S[5,5] + 44~\S[6,1,1,1,1] + 166~\S[6,2,1,1] 
+ 131~\S[6,2,2] + \\
217~\S[6,3,1] + 90~\S[6,4] + 47~\S[7,1,1,1] + 118~\S[7,2,1] + 
75~\S[7,3] + 29~\S[8,1,1] + 35~\S[8,2] + 9~\S[9,1] + ~\S[10]
$.
\end{small}

\begin{remark} Note that all timings that have been presented so far, display the computing
time for the system command \textsf{nc$\_$hilb} of \singular. This can be used directly (without
\textsf{nchilb}) if a finite set of monomials is provided in the letterplace format as an input.    
\end{remark}

\section{Conclusions and further directions}
\label{conclusion}

We believe that the two previous sections clearly show the power and
flexibility of the proposed approach to the computation of Hilbert series
for noncommutative structures. We plan to further extend our implementation
from the case of algebras to the case of (finitely generated) right modules
according to Theorem \ref{macth} and Proposition \ref{cycdec}. Researchers
who have interest also in representation theory may obtain essential information
about the decomposition of a $\GL_n(\KK)$-invariant algebra in terms of its
simple $\GL_n(\KK)$-submodules by combining our algorithms with procedures
performing the Schur function decomposition. Moreover, with fast approximations
of non-rational Hilbert series at hand, one can venture to enter
the intriguing realm of such functions. Finally, note that the iterative design
of our algorithms is immediately applicable also to the commutative case and
automata theory provides optimality in the number of iterations. Therefore,
we suggest to develop commutative variants of the proposed methods and to compare
them with the many existing implementations of commutative Hilbert series.

\section{Acknowledgements}
The authors would like to express their gratitude to the \singular team for
allowing full and guided access to the kernel of system. In particular,
we want to thank Hans Sch\"onemann for his cooperation on the implementation
of the algorithms. We would also like to thank Claus Fieker for fruitful discussions
and his assistance in using \magma. Finally, the authors gratefully acknowledge
Monica Lazzo and the anonymous reviewers for the careful reading of our manuscript and the valuable comments.

\end{document}